\documentclass[11pt]{amsart}
\usepackage[margin=1.2in]{geometry}
\usepackage{amsmath,amsthm,amssymb,mathtools,amsfonts,amsopn,amscd}
\usepackage{bm}
\usepackage{tikz}
\usepackage{tikz-cd}    
\usepackage{rotating}
\usepackage{commath}    
\usepackage{mathtools}    
\usepackage{graphicx}
\usepackage{etoolbox}
\usepackage{enumitem}
\usepackage{comment}
\usepackage{stmaryrd}
\usepackage{hyperref}
\usepackage{subfiles}
\usepackage{mathrsfs}
\usepackage{calc}
\usepackage[normalem]{ulem}
\usepackage[numbered]{bookmark}   

\setlist[enumerate,1]{label=(\roman*)}

\usepackage[
backend=biber,
style=alphabetic,
maxalphanames=10,
maxbibnames=10,
sorting=anyt,
]{biblatex}

\hypersetup{
	colorlinks=true,
	linkcolor=blue,
	urlcolor=black,
	citecolor=red,
	linktoc=page,
}

\theoremstyle{plain}
\newtheorem{thm}{Theorem}[section]
\newtheorem{lem}[thm]{Lemma}
\newtheorem{prop}[thm]{Proposition}

\theoremstyle{definition}
\newtheorem{df}[thm]{Definition}

\theoremstyle{remark}
\newtheorem{rem}[thm]{Remark}

\newcommand{\C}{\mathbb{C}}
\newcommand{\F}{\mathbb{F}}
\newcommand{\bbN}{\mathbb{N}}
\newcommand{\ZZ}{\mathbb{Z}}
\newcommand{\NN}{\mathbb{N}}
\newcommand{\QQ}{\mathbb{Q}}

\newcommand{\CC}{\mathbb{C}}

\newcommand{\FF}{\mathbb{F}}

\newcommand{\PP}{\mathbb{P}}
\newcommand{\Fp}{\FF_p}

\newcommand{\afk}{\mathfrak{a}}
\newcommand{\bfk}{\mathfrak{b}}
\newcommand{\cfk}{\mathfrak{c}}

\newcommand{\Afk}{\mathfrak{A}}

\newcommand{\Ecal}{\mathcal{E}}

\newcommand{\Ocal}{\mathcal{O}}

\newcommand{\Rcal}{\mathcal{R}}
\newcommand{\Scal}{\mathcal{S}}
\newcommand{\Tcal}{\mathcal{T}}

\newcommand{\ovl}{\overline}

\let\oldforall\forall
\renewcommand{\forall}{\oldforall \: }
\let\oldexist\exists
\renewcommand{\exists}{\oldexist \: }

\let\emptyset\varnothing

\usepackage{stmaryrd}
\SetSymbolFont{stmry}{bold}{U}{stmry}{m}{n}

\newcommand{\Fq}{\FF_q}
\newcommand{\Fqst}{\FF_q^\times}

\newcommand{\T}{\theta}

\newcommand{\dep}{\operatorname{dep}}
\newcommand{\wt}{\operatorname{wt}}

\let\l\ell

\newcommand{\ww}[1]{\underline{#1}} 
\newcommand{\bfG}{\widehat{G}}
\newcommand{\bfe}{\widehat{e}}
\newcommand{\bfE}{\widehat{E}}
\newcommand{\bfzeta}{\widehat{\zeta}}
\newcommand{\bff}{\bm{f}}
\newcommand{\bfg}{\bm{g}}
\newcommand{\bfS}{\widehat{S}}

\newcommand{\bfz}{\mathbf{z}}
\newcommand{\bfy}{\mathbf{y}}
\newcommand{\Lamz}{\Lambda_{\bfz'}}

\allowdisplaybreaks

\newcommand{\ang}[1]{\langle #1 \rangle}

\makeatletter

\newcommand{\myToC}{{
		\renewcommand{\contentsname}{}
		\@starttoc{toc}{\contentsname}
}}

\patchcmd{\@tocline}
{\hfil}
{\leaders\hbox{\,.\,}\hfil}

\makeatother

\title[On $q$-Shuffle Relations for Multiple Eisenstein Series of Arbitrary Rank]{On $q$-Shuffle Relations for Multiple Eisenstein Series of Arbitrary Rank in Positive Characteristic}
\author{Ting-Wei Chang, Song-Yun Chen, Fei-Jun Huang, and Hung-Chun Tsui}
\date{\today}

\thanks{The first author is supported by the National Science and Technology Council grant no. 109-2115-M-007-017-MY5 and 113-2628-M-007-003.
	The latter three authors are supported by NSTC Grant Number 113-2628-M-007-004}

\begin{document}
	
\begin{abstract}
    In this paper, we define the multiple Eisenstein series of arbitrary rank in positive characteristic, with Thakur's multiple zeta values appearing as the “constant terms" of their expansions in terms of “multiple Goss sums”.
    We show that the multiple Eisenstein series satisfy the same $q$-shuffle relations as the multiple zeta values do, thereby lifting the relations from “values” to “functions”.
\end{abstract}

\maketitle
\tableofcontents

\section{Introduction}

\subsection{Multiple Zeta Values in Positive Characteristic}

Let $p$ be a prime number and $q$ be a power of $p$.
Let $A:=\Fq[\T]$ be the polynomial ring in the variable $\T$ over a finite field $\Fq$ of $q$ elements (with the prime field $\Fp$), and $K := \Fq(\T)$ be its field of fractions.
We let $K_\infty$ be the completion of $K$ with respect to the absolute value $|\cdot|$ at the infinite place normalized so that $|\T| = q$, and $\CC_\infty$ be the completion of a fixed algebraic closure of $K_\infty$.
For a positive integer $d$, we let $A_+,A_{+,d},A_{<d}$ be the subsets of $A$ consisting of all monic polynomials, monic polynomials of degree $d$, and polynomials of degree less than $d$, respectively.
Moreover, we define $A_{+,0}=\{1\}$.

In \cite{thakur2004function}, Thakur introduced the multiple zeta values in positive characteristic, which generalized the notion of Carlitz single zeta values \cite{carlitz1935oncertain} as follows. Let $\NN$ denote the set of positive integers.
For any $d\in\ZZ_{\ge 0}$ and non-empty index $\ww{a} = (a_1,\ldots,a_m)\in \NN^m$, we define the \textit{power sum}
\begin{equation} \label{eq-power-sum}
    S_d(\ww{a}):= \sum_{ \substack{f_1,\ldots,f_m\in A_{+} \\ d=\deg f_1>\cdots>\deg f_m\ge 0}} \frac{1}{f_1^{a_1}\cdots f_m^{a_m}}\in K
\end{equation}
and Thakur's \textit{multiple zeta value} 
\[
\zeta_A(\ww{a}) := \sum_{d=0}^\infty S_d(\ww{a}) = \sum_{\substack{f_1,\ldots,f_m\in A_+ \\ \deg f_1 > \cdots > \deg f_m \geq 0}}
\frac{1}{f_1^{a_1} \cdots f_m^{a_m}} \in K_\infty.
\]
Throughout this paper, the term “multiple zeta values” refers to Thakur’s multiple zeta values, unless otherwise specified.
The quantities $\wt(\ww{a}) := \sum_{i=1}^{m} a_i$ and $\dep(\ww{a}) := m$ are called the weight and depth of the presentation $\zeta_A(\ww{a})$, respectively.
It is known by \cite{thakur2009finite} that all these multiple zeta values are non-vanishing.
Moreover, Thakur \cite{thakur2010shuffle} showed the existence that the product of two multiple zeta values can be expressed as an $\Fp$-linear combination of multiple zeta values of the same weight, called the \textit{$q$-shuffle relations}.
In particular, this implies that the $\Fp$-vector space spanned by all multiple zeta values admits an $\Fp$-algebra structure. 
Furthermore, it was shown by Chang \cite{chang2014linear} that the $\ovl{K}$-algebra spanned by all multiple zeta values is a graded algebra, where $\ovl{K}$ is the algebraic closure of $K$ in $\CC_\infty$.

Regarding the existence of the $q$-shuffle product mentioned above, the first concrete example was due to Chen \cite{Chen2015}, who proved the following explicit formula:
For any $a,b \in \NN$,
\[
\zeta_A(a)\zeta_A(b)=\zeta_A(a,b)+\zeta_A(b,a)+\zeta_A(a+b)+\sum_{\substack{i+j=a+b \\ q-1\mid j}}\Delta^{i,j}_{a,b}\zeta_A(i,j),
\]
where $\Delta^{i,j}_{a,b} := (-1)^{a-1}\binom{j-1}{a-1}+(-1)^{b-1}\binom{j-1}{b-1}$ and $\binom{m}{n}$ denotes the usual binomial coefficient modulo $p$.
Later on, based on Chen's formula, Yamamoto gave the following (conjectural) explicit $q$-shuffle relations for the multiple zeta values, which were verified in Shi's PhD thesis \cite{Shi2018}.

\begin{df}\label{df-zeta-values-as-words}
    Let $\Scal$ be the free monoid generated by the set $\{x_k \mid k \in \NN\}$ and let $\Rcal$ be the $\FF_p$-vector space generated by $\Scal$. For a non-empty index $\ww{a}=(a_1,\ldots,a_m)\in\NN^m$, we define $x_{\ww{a}}:=x_{a_1}\cdots x_{a_m}$, and for $\ww{a}=\emptyset$, we define $x_\emptyset:=1$.
    An element $x_{\ww{a}}$ is called a \textit{word of depth $m$}.
    When $m \geq 2$, we put $x_{\ww{a}^-}:=x_{a_2}\cdots x_{a_m}$, and when $m=1$, we put $x_{\ww{a}^-}:=x_{\emptyset} = 1$.
    We define the \textit{$q$-shuffle product} $\ast$ on $\Rcal$ inductively on the sum of depths as follows:
    \begin{enumerate}
        \item For the empty word $x_\emptyset=1$ and any $\afk\in\Rcal$, define
        \[
        1*\afk=\afk*1=\afk.
        \]
        \item For non-empty indices $\ww{a}=(a_1,\ldots,a_m)\in\NN^m$ and $\ww{b}=(b_1,\ldots,b_n)\in\NN^n$, define
        \begin{multline*}
            x_{\ww{a}} * x_{\ww{b}}
            = x_{a_1} ( x_{\ww{a}^-} * x_{\ww{b}} ) 
            + x_{b_1} ( x_{\ww{a}} * x_{\ww{b}^-} ) 
            + x_{a_1 + b_1} ( x_{\ww{a}^-} * x_{\ww{b}^-} ) \\
            + \sum_{\substack{i+j = a_1 + b_1 \\ q-1 \mid  j}} 
            \Delta^{i,j}_{a_1,b_1} x_i ( ( x_{\ww{a}^-} * x_{\ww{b}^-} ) * x_j )
        \end{multline*}
        where
        \begin{enumerate}
            \item $\Delta^{i,j}_{a,b} := (-1)^{a-1}\binom{j-1}{a-1}+(-1)^{b-1}\binom{j-1}{b-1}$.
            \item $x \sum_i \epsilon_ix_i := \sum_i \epsilon_i (xx_i)$ for $\epsilon_i \in \Fp$ and $x,x_i \in \Scal$.
        \end{enumerate}
        \item Expand the product $\ast$ to the $\Fp$-vector space $\Rcal$ by the distributive law.
    \end{enumerate}
\end{df}

\begin{thm}[\cite{Shi2018}]\label{thm-zeta-alg-hom}
    For $d\in \NN$, let $\bfS_{<d}, \bfzeta_A:\Rcal \to K_\infty$ be the unique $\F_p$-linear maps satisfying
    \[
    \bfS_{<d}(1) := 1,
    \quad
    \bfS_{<d}(x_{\ww{a}})
    := S_{<d}(\ww{a}) 
    := \sum_{i=0}^{d-1} S_i(\ww{a}),
    \]
    and
    \[
    \bfzeta_A(1):=1,
    \quad
    \bfzeta_A (x_{\ww{a}})
    := \lim_{d\to \infty} \bfS_{<d}(x_{\ww{a}})
    = \zeta_A(\ww{a}).
    \]
    Then $\bfS_{<d}$ is an $\F_p$-algebra homomorphism.
    Consequently, $\bfzeta_A$ is also an $\F_p$-algebra homomorphism, i.e., 
    \[
    \bfzeta_A(x_{\ww{a}} \ast x_{\ww{b}})=\zeta_A(\ww{a})\zeta_A(\ww{b}).
    \]
\end{thm}

\subsection{Multiple Eisenstein Series}

We now define the multiple Eisenstein series of arbitrary rank and state the main result of this paper.
For a positive integer $r$, let $\Omega^r$ be the Drinfeld symmetric space of rank $r$ consisting of elements in $\PP^{r-1}(\CC_\infty)$ with $K_{\infty}$-linearly independent entries.
It is known that $\Omega^r$ admits a rigid analytic structure \cite[Proposition 6.1]{Drinfeld1974}.
Each element $\bfz \in \Omega^r$ can be identified as $\bfz = (z_1,\ldots,z_{r-1},1) \in \CC_\infty^r$.
Hence, we may view $\Omega^r$ as the following subset of $\C_\infty^r$:
\[
\Omega^r = \{(z_1,\ldots,z_r)\in\C_\infty^r:z_1,\ldots,z_r \text{ are }K_\infty\text{-linearly independent and }z_r=1\}.
\]
For $\bff = (f_1,\ldots,f_r)\in A^r$ and $\bfz = (z_1,\ldots,z_r) \in \Omega^r$, we put
\[
\ang{\bff,\bfz} := \sum_{i=1}^r f_iz_i \in \CC_\infty.
\]

The following definition generalizes the rank two case due to Chen \cite{Chen2017}.

\begin{df}
    Let $r\geq 1$, $\bfz = (z_1,\ldots,z_{r-1},z_r=1) \in \Omega^r$ and  $\bff,\bfg\in A^r$.
    Write $\bff=(f_1,\ldots,f_r)$ and $\bfg=(g_1,\ldots,g_r)$.
    We define a partial order $\succ$ on $Az_1+ \cdots +Az_{r-1}+A$ as follows:
    \begin{enumerate}
        \item $\ang{\bff,\bfz}\succ 0$ if $f_1=\cdots=f_{i-1}=0$ and $f_i\in A_+$ for some $1\leq i\leq r$.
        \item For $\ang{\bff,\bfz}, \ang{\bfg,\bfz}\succ 0$, we write $\ang{\bff,\bfz} \succ \ang{\bfg,\bfz}$ if one of the following holds:
        \begin{enumerate}
            \item There exists $1\leq i\leq r$ such that
            $f_1=\cdots=f_{i-1}=g_1=\cdots=g_{i-1}=0$ and
            $f_i,g_i\in A_+$ with $\deg f_i>\deg g_i$.
            \item There exists $1\leq i\leq r$ such that
            $f_1=\cdots=f_{i-1}=g_1=\cdots=g_{i-1}=0$ and
            $f_i\in A_+,g_i=0$.
        \end{enumerate}
    \end{enumerate}
\end{df}

We now define the multiple Eisenstein series of rank $r$, which generalize the single and double Eisenstein series on $\Omega^2$ defined by Chen \cite{Chen2017}.

\begin{df}[Multiple Eisenstein series]\label{df-MES-positive-charcteristic}
    Let $r\geq 1$. For any non-empty index $\ww{a}=(a_1,\ldots,a_m)\in\bbN^m$, we define the \textit{multiple Eisenstein series of rank $r$} on $\Omega^r$ by 
    \[
    E_r(\ww{a};\bfz)
    := \sum_{\substack{\bff_1,\ldots,\bff_m \in A^r \\ \ang{\bff_1,\bfz} \succ \cdots \succ \ang{\bff_m,\bfz} \succ 0}}
    \frac{1}{\ang{\bff_1,\bfz}^{a_1}\cdots \ang{\bff_m,\bfz}^{a_m}}.
    \]
    By convention, we put   $E_r(\emptyset;\mathbf{z}):=1$.
\end{df}

Let $\Ocal(\Omega^r)$ be the algebra of rigid analytic functions on $\Omega^r$.
We will see in Proposition \ref{prop-MES-rigid} that the multiple Eisenstein series of rank $r$ are rigid analytic functions on $\Omega^r$, i.e., $E_r(\ww{a};\bfz) \in \Ocal(\Omega^r)$.

\begin{df}  \label{df-E-hat}
    For $r\ge 1$, we define $\bfE_r: \Rcal \to \Ocal(\Omega^r)$ to be the unique $\FF_p$-linear map satisfying  
    \[
    \bfE_r(1) := 1
    \quad
    \text{and}
    \quad
    \bfE_r(x_{\ww{a}}) := E_r(\ww{a};\bfz).
    \]
\end{df}

Our main result is stated as follows.

\begin{thm}[Main Result] \label{thm-main-theorem}
    For $r\geq 1$, $\bfE_r$ is an $\FF_p$-algebra homomorphism, i.e.,
    \[
    \bfE_r(x_{\ww{a}} \ast x_{\ww{b}}) = E_r(\ww{a};\bfz) E_r(\ww{b};\bfz).
    \]
\end{thm}

We notice that by our identification, $\Omega^1=\{1\}$ and
\[
E_1(\ww{a};\mathbf{z})=\sum_{\substack{f_1,\ldots,f_m\in A \\ \left\langle f_1,\mathbf{z}\right\rangle \succ \cdots \succ \left\langle f_m,\mathbf{z}\right\rangle\succ 0}}\frac{1}{\left\langle f_1,\mathbf{z}\right\rangle^{a_1}\cdots \left\langle f_m,\mathbf{z}\right\rangle^{a_m}}= \sum_{\substack{f_1,\ldots,f_m\in A_+ \\ \deg f_1 > \cdots > \deg f_m \geq 0}}
\frac{1}{f_1^{a_1} \cdots f_m^{a_m}} =\zeta_A(\ww{a}).
\]
That is, multiple Eisenstein series degenerate to multiple zeta values and Theorem \ref{thm-zeta-alg-hom} serves as the rank one case of our main result.

In Proposition \ref{prop-expansion-MES}, we prove that for any non-empty index $\ww{a}\in \NN^{m}$, the multiple zeta value $\zeta_A(\ww{a})$ appears as the “constant term" of the expansion of $E_r(\ww{a};\mathbf{z})$ in terms of “multiple Goss sums” (see Definition \ref{df-Goss-sum}).
Hence, Theorem \ref{thm-main-theorem} illustrates the phenomenon that Thakur's $q$-shuffle relations for the multiple zeta values can be “lifted” to the corresponding multiple Eisenstein series, thereby generalizing Chen's result \cite{Chen2017} to arbitrary depth and arbitrary rank.

In the classical situation, the notion of multiple Eisenstein series on the complex upper half-plane was initiated by Gangl, Kaneko, and Zagier \cite{gkz2006double} for the depth two case, and generalized by Bachmann \cite{bachmann2012multiple} to arbitrary depth.
It is known that, the product of two real multiple zeta values can be expressed as $\QQ$-linear combinations of real multiple zeta values of the same weight via their integral (resp. series) representations, called the shuffle (resp. stuffle) relations.
These relations are further “lifted” to multiple Eisenstein series (in one variable) in the works of \cite{bt2018double} and \cite{bachmann2019algebra}, corresponding to the rank two case of Theorem \ref{thm-main-theorem}. For recent developments and references, we refer the readers to Bachmann’s papers.

In contrast to the classical situation, Thakur \cite{thakur2009finite} conjectured that the multiple zeta values in positive characteristic admit only one “shuffle relation" and gave some heuristic reasons and observations based on the “motivic interpretations" of multiple zeta values \cite{anderson2009multizeta}.
Therefore, Theorem \ref{thm-main-theorem} completely resolves the question of lifting shuffle relations for multiple zeta values to multiple Eisenstein series in the existing literature.

\subsection{Ingredients of the Proof}
  
We first establish the expansion of the multiple Eisenstein series of rank $r\ge 2$ in terms of rank $r$ multiple Goss sums and rank $r-1$ multiple Eisenstein series (see Proposition \ref{prop-expansion-MES} and Definition \ref{df-Goss-sum}).
This allows us to prove Theorem \ref{thm-main-theorem} by induction on the rank $r$, taking Theorem \ref{thm-zeta-alg-hom} as the base case.
We consider the Goss power sums (see Definition \ref{df-Goss-sum}), which are truncations of multiple Goss sums and serve as a higher rank analog of power sums.
Inspired by Chen's calculation in \cite[Theorem 4.1]{Chen2017}, we prove an explicit formula for the product of depth one Goss power sums (see Proposition \ref{prop-G_d(r)G_d(s)}) and further formulate the $q$-shuffle relations for multiple Goss sums.

Next, we define the $q$-shuffle algebra $\Ecal$ (see Definition \ref{df-Ecal}), which encodes the $q$-shuffle relations for multiple Goss sums.
This enables us to decompose the map $\bfE_r:\Rcal\to \mathcal{O}(\Omega^r)$ into
\[
\bfE_r=\bfG_r \circ \bfe : \Rcal\to \Ecal/\Afk \to \mathcal{O}(\Omega^r),
\] 
where $\Afk$ denotes the associator ideal of $\Ecal$, $\bfG_r$ is the realization map of rank $r$ multiple Goss sums (see Definition \ref{df-bfG}), and $\bfe$ is the formalization of the expansion of multiple Eisenstein series (see Definition \ref{df-bfe}). 
The main theorem (Theorem \ref{thm-main-theorem}) is now reduced to showing that the maps $\bfG_r$ and $\bfe$ are $\Fp$-algebra homomorphisms. 
To this end, we divide the proof into the following two parts:

\begin{enumerate}
    \item Inspired by the analogy between Goss power sums and Thakur's power sums, we adopt an approach similar to Shi’s proof of Theorem~\ref{thm-zeta-alg-hom}.
    Under the induction hypothesis, we show that the realization map
    \[
    \bfG_r : \Ecal \to \Ocal(\Omega^r)
    \]
    is an $\Fp$-algebra homomorphism (see Proposition~\ref{prop-bfG-ring-hom}), which is equivalent to saying that the rank $r$ multiple Goss sums satisfy the $q$-shuffle relations formulated earlier.
    \item We prove that the map $\bfe$ satisfies explicit “$q$-shuffle relations” (see Proposition \ref{prop-ye(x)*ye(x)}) using the associativity of $\Ecal/\Afk$.
    (We conjecture that the algebra $\Ecal$ is associative in Remark \ref{rmk-ass}.)
    This, in turn, implies that $\bfe$ is an $\FF_p$-algebra homomorphism (see Proposition \ref{prop-bfe-alg-hom}).
\end{enumerate}

\section{Multiple Eisenstein Series and Multiple Goss Sums}

\subsection{Basic Properties of Multiple Eisenstein Series}

For a $K_\infty$-rational hyperplane $H$ in $\PP^{r-1}(\CC_\infty)$, we choose any unimodular linear form
\[
\l_H(X_1,\ldots,X_r) := c_1 X_1 + \cdots + c_r X_r \in K_\infty[X_1,\ldots,X_r],
\quad
\max\{|c_i| : 1 \leq i \leq r\} = 1,
\]
to be the defining equation of $H$.
Note that for $\bfz = (z_1, \ldots, z_{r-1}, z_r = 1)\in\Omega^r$, $|\l_H(\bfz)|$ is well-defined and non-zero.
For $\bfy = (y_1,\ldots,y_r) \in \CC_\infty^r$, put $|\bfy| := \max\{|y_i| : 1 \leq i \leq r\}$.
Set
\[
h(\bfz) := \frac{1}{|\bfz|} \inf\{|\l_H(\bfz)| : H \text{ is a } K_\infty\text{-rational hyperplane}\}.
\]
Recall that
\[
\Omega^r = \bigcup_{n \in \NN} \Omega^r_n
\quad
\text{where}
\quad
\Omega^r_n := \{\bfz \in \Omega^r : h(\bfz) \geq q^{-n}\}
\]
is an admissible cover of $\Omega^r$ (see \cite[Proposition 3.4]{bbp2024drinfeld}), and a function $f : \Omega^r \to \CC_\infty$ is called rigid analytic if its restriction to each $\Omega^r_n$ is the uniform limit of a sequence of rational functions without poles on $\Omega^r_n$ (see \cite[Definition 2.2.1]{fresnelvanderput2004rigid}).

\begin{prop}\label{prop-MES-rigid}
    Let $r\ge 1$.
    For any index $\ww{a}$, the multiple Eisenstein series $E_r(\ww{a};\bfz): \Omega^r \to \CC_\infty$ is a rigid analytic function on $\Omega^r$.
\end{prop}

\begin{proof}
    When $r=1$ or $\ww{a}$ is empty, the result is clear.
    Therefore, we may assume $r\ge 2$ and $\ww{a}=(a_1,\ldots,a_m)\in \bbN^m$ is non-empty.
    For any non-zero $\bff = (f_1,\ldots,f_r) \in A^r$ and any $\bfz=(z_1,\ldots,z_{r-1}, z_r=1) \in \Omega_n^r$, we have
    \[
    h(\bfz) \leq \frac{1}{|\bfz|} \left| \frac{f_1}{|\bff|} z_1 + \cdots + \frac{f_r}{|\bff|} z_r\right|.
    \]
    Since $\bfz \in \Omega_n^r$ and $|\bfz| \geq 1$ by our identification, it follows that
    \[
    |\ang{\bff,\bfz}|
    = |f_1z_1 + \cdots + f_rz_r| 
    \geq h(\bfz) |\bfz| |\bff| 
    \geq q^{-n} |\bff|.
    \]
    Thus, for $\bff_1,\ldots,\bff_m \in A^r$ with $\ang{\bff_1,\bfz} \succ \cdots \succ \ang{\bff_m,\bfz} \succ 0$, we have
    \[
    |\ang{\bff_1,\bfz}^{a_1}\cdots \ang{\bff_m,\bfz}^{a_m}|
    \geq q^{-n \wt(\ww{a})} |\bff_1|^{a_1} \cdots |\bff_m|^{a_m}.
    \]
    Hence,
    \[
    E_r(\ww{a};\bfz)
    = \sum_{\substack{\bff_1,\ldots,\bff_m \in A^r \\ \ang{\bff_1,\bfz} \succ \cdots \succ \ang{\bff_m,\bfz} \succ 0}}
    \frac{1}{\ang{\bff_1,\bfz}^{a_1}\cdots \ang{\bff_m,\bfz}^{a_m}}
    \]
    is the uniform limit of rational functions with poles outside $\Omega_n^r$.
\end{proof}

In what follows, we define three types of notation: Given an index $\ww{a}$, we write $\ww{a}^{(i)}$ for the index obtained by deleting its first $i$ coordinates, and $\ww{a}_{(i)}$ for the one obtained by keeping its first $i$ coordinates. Similarly, given $\bfz \in \Omega^r$, we let $\bfz' \in \Omega^{r-1}$ denote the element obtained by deleting its first coordinate.

\begin{df}
    \begin{enumerate}
        \item For any non-empty index $\ww{a}=(a_1,\ldots,a_m)\in\bbN^m$ and $0\leq i\leq m-1$, we define
        \[
        \ww{a}^{(i)} := (a_{i+1},\ldots,a_m)
        \quad
        \text{and}
        \quad
        \ww{a}^{(m)} := \emptyset.
        \]
        On the other hand, for each $1\le i\le m$, we also define
        \[
        \ww{a}_{(i)} := (a_1,\ldots,a_i)
        \quad
        \text{and}
        \quad
        \ww{a}_{(0)} := \emptyset.
        \]
        By convention, we put
        \[
        \emptyset^{(i)} = \emptyset_{(i)} := \emptyset
        \]
        for all $i\geq 0$.

        \item For $r\ge 2$ and $\bfz = (z_1,\ldots,z_{r-1},z_r=1)\in \Omega^r$, let $\bfz' : = (z_2,\ldots,z_r)\in \Omega^{r-1}$.
    \end{enumerate}
\end{df}

For $r\ge 1 $ and $\bfz = (z_1,\ldots,z_{r-1},z_r=1) \in \Omega^r$ with the associated $A$-lattice $\Lambda_\bfz := A z_1 + \cdots + A z_r$ in $\CC_\infty$, we consider the function
\[
t_{\Lambda_\bfz}(x) := \sum_{\lambda \in \Lambda_\bfz} \frac{1}{x+\lambda} = \exp_{\Lambda_\bfz}(x)^{-1},
\]
where
\[
\exp_{\Lambda_\bfz}(x) = x\prod_{0 \neq \lambda \in \Lambda_\bfz} \left(1-\frac{x}{\lambda}\right)
\]
is the exponential function associated to $\Lambda_\bfz$ (see \cite[Chapter 4]{goss1996basic}).
Recall that in \cite{Gos1980} (see \cite{gekeler1988coefficients} also), Goss introduced the so-called \textit{Goss polynomials} $\{G_k^{\Lambda_{\mathbf{z}}}(x) : k\ge 1 \} \subseteq \CC_\infty[x]$  satisfying
\begin{equation}  \label{eq-Goss-polynomials}
    G^{\Lambda_\bfz}_k(t_{\Lambda_\bfz}(x)) = \sum_{\lambda \in \Lambda_\bfz} \frac{1}{(x+\lambda)^k}.
\end{equation}

With the necessary notions introduced, we can now state the following expression for the multiple Eisenstein series.

\begin{prop} \label{prop-expansion-MES}
    Let $r\ge 2$.
    For any non-empty index $\ww{a} = (a_1,\ldots,a_m) \in \bbN^m$ and $\bfz \in \Omega^r$, we have 
    \begin{multline*}
        E_r(\ww{a};\bfz) = E_{r-1}(\ww{a};\bfz') \\
        + \sum_{i=1}^{m} E_{r-1}(\ww{a}^{(i)};\bfz') \sum_{\substack{f_1,\ldots,f_i\in A_+ \\ \deg f_1>\cdots>\deg f_i\ge 0}} G_{a_1}^{\Lambda_{\bfz'}}(t_{\Lambda_{\bfz'}}(f_1z_1)) \cdots G_{a_i}^{\Lambda_{\bfz'}}(t_{\Lambda_{\bfz'}}(f_iz_1)).
    \end{multline*}
\end{prop}

\begin{proof}
    For $\bfg=(g_1,\ldots,g_r)\in A^r$, we put $\bfg' := (g_2,\ldots,g_r)$.
    For each $0\le i\le m$, let $\mathscr{A}_i$ be the subset of $(A^r)^m$ consisting of elements $(\bff_1,\ldots,\bff_m)\in (A^r)^m$ with $\bff_j=(f_{j,1},\ldots,f_{j,r})$ satisfying the following conditions: 
    \begin{enumerate} 
        \item $f_{1,1},\ldots,f_{i,1}\in A_+$,
        \item $\deg f_{1,1} > \cdots > \deg f_{i,1}\ge 0$,
        \item $f_{i+1,1}=\cdots=f_{m,1}=0$,
        \item $\ang{\bff_{i+1}',\bfz'} \succ \cdots \succ \ang{\bff_m',\bfz'}$.
    \end{enumerate}
    Summing over $\mathscr{A}_i$, one sees that
    \begin{align*}
        & \sum_{(\bff_1,\ldots,\bff_m)\in \mathscr{A}_i}\frac{1}{\ang{ \bff_1,\bfz}^{a_1}\cdots \ang{ \bff_m,\bfz}^{a_m}}
        \\
        ={} &\begin{multlined}[t]
            \left( \sum_{\substack{\bff_1,\ldots,\bff_i\in A_+\times A^{r-1} \\ \deg f_{1,1}>\cdots>\deg f_{i,1}\ge 0}}\frac{1}{(f_{1,1}z_1+\ang{ \bff_1',\bfz'})^{a_1}\cdots (f_{i,1}z_1+\ang{ \bff_i',\bfz'})^{a_i}} \right) \\
            \left( \sum_{\substack{\bff_{i+1},\ldots,\bff_m\in A^{r-1} \\ \ang{ \bff_{i+1},\bfz'} \succ \cdots \succ \ang{ \bff_m,\bfz'} \succ 0}}\frac{1}{\ang{ \bff_{i+1},\bfz'}^{a_{i+1}}\cdots \ang{\bff_m,\bfz'}^{a_m}} \right)
        \end{multlined} \\
        ={} &\left( \sum_{\substack{f_{1,1},\ldots,f_{i,1} \in A_+ \\ \deg f_{1,1}>\cdots>\deg f_{i,1}\ge 0}}G_{a_1}^{\Lambda_{\mathbf{z}'}}(t_{\Lambda_{\mathbf{z}'}}(f_{1,1}z_1)) \cdots G_{a_i}^{\Lambda_{\mathbf{z}'}}(t_{\Lambda_{\mathbf{z}'}}(f_{i,1}z_1))\right) \cdot E_{r-1}(\ww{a}^{(i)};\bfz')
    \end{align*}
    by \eqref{eq-Goss-polynomials} and Definition \ref{df-MES-positive-charcteristic}.
    Since the set 
    \[
    \{(\bff_1,\ldots,\bff_m)\in (A^r)^m:\left\langle \bff_1,\bfz\right\rangle \succ \cdots \succ \left\langle \bff_m,\bfz\right\rangle \succ 0\}
    \]
    can be decomposed as a disjoint union of $\mathscr{A}_i$ for $0\leq i\leq m$, the result follows.
\end{proof}

In particular, for $\bfz = (z,1)\in \Omega^2$, the expansion of $E_2(\ww{a}; (z,1))$ is
\[
E_2({\ww{a}};(z,1)) = \zeta_A(\ww{a}) + \sum_{i=1}^{m} \zeta_A(\ww{a}^{(i)})\sum_{\substack{f_1,\ldots,f_i\in A_+ \\ \deg f_1>\cdots>\deg f_i\ge 0}}G_{a_1}^{\Lambda_{1}}(t_{\Lambda_{1}}(f_1z))\cdots G_{a_i}^{\Lambda_{1}}(t_{\Lambda_{1}}(f_iz)).
\]
Inductively, we see that the multiple zeta value in question appears as the constant term (in terms of Goss polynomials) of the corresponding multiple Eisenstein series of arbitrary rank.

\subsection{Goss Power Sums and Multiple Goss Sums} 

\begin{df}  \label{df-Goss-sum}
    For any non-empty index $\ww{a} = (a_1, \ldots, a_m) \in \NN^m$ and $d \in \ZZ_{\ge 0}$, we define the \textit{Goss power sum} by
    \begin{equation}   \label{eq-Goss-power-sum}
        G_r^{=d} (\ww{a} ; \bfz) 
        := \sum_{\substack{f_1, \ldots ,f_m \in A_+ \\ d=\deg f_1>\cdots>\deg f_m \ge 0}} 
        G^{\Lamz}_{a_1} (t_{\Lamz}(f_1z_1)) \cdots G^{\Lamz}_{a_m} (t_{\Lamz}(f_mz_1))
    \end{equation}
    where $\bfz=(z_1,\ldots,z_{r-1},z_r=1)\in \Omega^r$.
    Moreover, we define for $d\in \NN$,
    \[
    G_r^{<d}(\ww{a} ; \bfz) := \sum_{i=0}^{d-1} G_r^{=i}(\ww{a} ; \bfz)
    \quad \text{and} \quad 
    G_r(\ww{a} ; \bfz) := \lim_{d\to \infty} G_{r}^{<d} (\ww{a} ; \bfz).
    \]
    The latter is called the \textit{multiple Goss sum}.
    We conventionally put $G_r^{=d}(\emptyset;\bfz)=\delta_{d,0}$ and $G_r^{<d}(\emptyset;\bfz)=1$, where $\delta_{i,j}$ denotes the Kronecker delta.
\end{df}

In view of \eqref{eq-Goss-polynomials},
we have
\begin{equation}   \label{eq-goss-power-sum-analog}
    G_r^{=d} (\ww{a} ; \bfz)
    = \sum_{\substack{f_1, \ldots ,f_m \in A_+ \\ d=\deg f_1>\cdots>\deg f_m \ge 0 \\ \bfg_1,\ldots,\bfg_m \in A^{r-1}}} 
    \frac{1}{(f_1z_1+\left\langle \bm{g}_1,\mathbf{z}'\right\rangle)^{a_1} \cdots (f_mz_1+\left\langle \bm{g}_m,\mathbf{z}'\right\rangle)^{a_m}}.
\end{equation}
Thus, the Goss power sum $G_r^{=d} (\ww{a} ; \bfz)$ serves as a higher rank analog of Thakur's power sum \eqref{eq-power-sum}.
Notice that as partial sums of the multiple Eisenstein series, $G_r^{=d}(\ww{a}; \bfz)$ and $G_r(\ww{a}; \bfz)$ are rigid analytic functions on $\Omega^r$ (recall Proposition \ref{prop-expansion-MES}).

Note that under this definition, Proposition \ref{prop-expansion-MES} can be restated as follows:
For any non-empty index $\ww{a} = (a_1,\ldots, a_m) \in \bbN^m$ and $\bfz \in \Omega^r$, we have
\begin{equation}    \label{eq-restate-t-expansion}
    E_r(\ww{a} ; \bfz) = 
    E_{r-1}(\ww{a}; \bfz') + 
    \sum_{i=1}^{m} 
    E_{r-1}(\ww{a}^{(i)}; \bfz') 
    G_r(\ww{a}_{(i)}; \bfz).
\end{equation}
We also note that for all $\ww{a}=(a_1,\ldots,a_m)\in\bbN^m$, we have
\begin{equation}  \label{eq-G_d=G_dG_<d}
    G_r^{=d}(\ww{a}; \bfz) = 
    G_r^{=d}(a_1; \bfz)
    G_r^{<d}(\ww{a}^{(1)}; \bfz).
\end{equation}

Recall that for $i,j,a,b \in \NN$, we set
\[
\Delta^{i,j}_{a,b} := (-1)^{a-1}\binom{j-1}{a-1}+(-1)^{b-1}\binom{j-1}{b-1} \in \Fp.
\]

\begin{prop}\label{prop-G_d(r)G_d(s)}
    Let $d\in\ZZ_{\ge 0}$. 
    For $a,b \in \NN$, we have
    \begin{multline*}
        G_r^{=d}(a; \bfz)G_r^{=d}(b; \bfz) \\
        = G_r^{=d}(a+b; \bfz)
        + \sum_{\substack{i+j=a+b \\ q-1\mid  j}}\Delta_{a,b}^{i,j} E_{r-1}(j; \bfz') G_r^{=d}(i; \bfz)
        + \sum_{\substack{i+j=a+b \\ q-1\mid  j}}\Delta^{i,j}_{a,b}G_r^{=d}((i,j); \bfz).
    \end{multline*}
    for all $\bfz \in \Omega^r$.
\end{prop}

\begin{rem}
    We view Proposition \ref{prop-G_d(r)G_d(s)} as a higher rank analog of
    \[
    S_d(a)S_d(b)=S_d(a+b)+\sum_{\substack{i+j=a+b \\ q-1 \mid j}}\Delta^{i,j}_{a,b}S_d(i,j)
    \]
    (see \cite[Remark 3.2]{Chen2015}).
\end{rem}

\begin{proof}[Proof of Proposition \ref{prop-G_d(r)G_d(s)}]
    Recall that for any two distinct non-zero elements $f,g$ in an integral domain, we have the following well-known partial fraction decomposition
    \begin{equation}    \label{eq-partial-frac}
        \frac{1}{f^a g^b}=\sum_{i+j=a+b}\frac{1}{(f-g)^j}\left( \frac{(-1)^b\binom{j-1}{b-1}}{f^i}+\frac{(-1)^{j-a}\binom{j-1}{a-1}}{g^i} \right).
    \end{equation}

    Consider
    \begin{align*}
        &G_r^{=d}(a; \bfz)G_r^{=d}(b; \bfz)
        =\sum_{f, g\in A_{+,d}} G^{\Lamz}_a(t_{\Lamz}(fz_1))G^{\Lamz}_b(t_{\Lamz}(gz_1)) \\
        ={} &\sum_{f\in A_{+,d}} G^{\Lamz}_a(t_{\Lamz}(fz_1))G^{\Lamz}_b(t_{\Lamz}(fz_1))
        +\sum_{\substack{f,g\in A_{+,d} \\ f\neq g} } G^{\Lamz}_a(t_{\Lamz}(fz_1))G^{\Lamz}_b(t_{\Lamz}(gz_1)).
    \end{align*}
    For the first summand, we have 
    \begin{align*}
        &\sum_{f\in A_{+,d}}G^{\Lamz}_a(t_{\Lamz}(fz_1))G^{\Lamz}_b(t_{\Lamz}(fz_1))  \\
        ={} &\sum_{\substack{f\in A_{+,d} \\ \bm{g},\bm{h}\in A^{r-1}}} \frac{1}{(fz_1+\left\langle \bm{g},\mathbf{z}'\right\rangle)^a(fz_1+\left\langle \bm{h},\mathbf{z}'\right\rangle)^b} \quad\text{(by \eqref{eq-Goss-polynomials})} \\
        ={} &\begin{multlined}[t]
            \sum_{\substack{f\in A_{+,d} \\ \bm{g}\in A^{r-1}}}\frac{1}{(fz_1+\left\langle \bm{g},\mathbf{z}'\right\rangle)^{a+b}}
            \\
          + \sum_{\substack{f\in A_{+,d} \\ \bm{g},\bm{h}\in A^{r-1}, \bm{g}\neq \bm{h}}} \sum_{i+j=a+b} \frac{1}{\left\langle \bm{g}-\bm{h},\mathbf{z}'\right\rangle^j}\left( \frac{(-1)^b\binom{j-1}{b-1}}{(fz_1+\left\langle \bm{g},\mathbf{z}'\right\rangle)^i} + \frac{(-1)^{j-a}\binom{j-1}{a-1}}{(fz_1+\left\langle \bm{h},\mathbf{z}'\right\rangle^i} \right)  \quad\text{(by \eqref{eq-partial-frac})}
        \end{multlined} \\
        ={} &\begin{multlined}[t] G_r^{=d}(a+b;\mathbf{z})
        + \sum_{\substack{f\in A_{+,d} \\ \bm{g},\bm{h}\in A^{r-1}, \bm{g}\neq 0}} \sum_{i+j=a+b} \frac{1}{\left\langle \bm{g},\mathbf{z}'\right\rangle^j}\left( \frac{(-1)^b\binom{j-1}{b-1}}{(fz_1+\left\langle \bm{h},\mathbf{z}'\right\rangle)^i}+\frac{(-1)^{j-a}\binom{j-1}{a-1}}{(fz_1+\left\langle \bm{h}-\bm{g},\mathbf{z}'\right\rangle)^i} \right) \\ \quad\text{(by \eqref{eq-goss-power-sum-analog})}.
        \end{multlined}
    \end{align*}
    For a fixed pair $(i,j)$, we have
    \begin{align*}
        \sum_{\substack{f\in A_{+,d} \\ \bm{g},\bm{h}\in A^{r-1}, \bm{g}\neq 0}} \frac{1}{\left\langle \bm{g},\mathbf{z}'\right\rangle^j} \frac{(-1)^b\binom{j-1}{b-1}}{(fz_1+\left\langle \bm{h},\mathbf{z}'\right\rangle)^i}
        &= \sum_{\epsilon\in\Fq^\times} \sum_{\substack{f\in A_{+,d} \\ \bm{g},\bm{h}\in A^{r-1}, \langle\bm{g}, \mathbf{z}'\rangle \succ 0}} \frac{1}{(\epsilon \langle \bm{g},\mathbf{z}'\rangle)^j} \frac{(-1)^b\binom{j-1}{b-1}}{(fz_1+\left\langle \bm{h},\mathbf{z}'\right\rangle)^i} \\
        &= \begin{cases}
            0, &\text{if } q-1\nmid j, \\
            \displaystyle-\sum_{\substack{f\in A_{+,d} \\ \bm{g}, \bm{h}\in A^{r-1}, \langle\bm{g}, \mathbf{z}'\rangle \succ 0}} \frac{1}{\left\langle \bm{g},\mathbf{z}'\right\rangle^j} \frac{(-1)^b\binom{j-1}{b-1}}{(fz_1+\left\langle \bm{h},\mathbf{z}'\right\rangle)^i}, &\text{if } q-1\mid  j.
        \end{cases}
    \end{align*}
    Here we use the fact that $\sum_{\epsilon \in \Fqst} \epsilon^j = 0$ if $q-1 \nmid j$ and $-1$ otherwise.
    Similarly,
    \begin{align*}
        &\sum_{\substack{f\in A_{+,d} \\ \bm{g},\bm{h}\in A^{r-1}, \bm{g}\neq 0}} \frac{1}{\left\langle \bm{g},\mathbf{z}'\right\rangle^j}\frac{(-1)^{j-a}\binom{j-1}{a-1}}{(fz_1+\left\langle \bm{h}-\bm{g},\mathbf{z}'\right\rangle)^i} \\
        ={} &\sum_{\substack{f\in A_{+,d} \\ \bm{g},\bm{h}\in A^{r-1}, \bm{g}\neq 0}} \frac{1}{\left\langle \bm{g},\mathbf{z}'\right\rangle^j}\frac{(-1)^{j-a}\binom{j-1}{a-1}}{(fz_1+\left\langle \bm{h},\mathbf{z}'\right\rangle)^i} \\
        ={} &\begin{cases}
            0, &\text{if } q-1\nmid j, \\
            \displaystyle-\sum_{\substack{f \in A_{+,d} \\ \bm{g},\bm{h}\in A^{r-1}, \langle\bm{g},\mathbf{z}'\rangle \succ 0}} \frac{1}{\left\langle \bm{g},\mathbf{z}'\right\rangle^j} \frac{(-1)^{j-a}\binom{j-1}{a-1}}{(fz_1+\left\langle \bm{h},\mathbf{z}'\right\rangle)^i}, &\text{if } q-1\mid  j.
        \end{cases}
    \end{align*}
    Therefore, the first summand becomes
    \begin{align*}
        & G_r^{=d}(a+b;\mathbf{z})
        + \sum_{\substack{f\in A_{+,d} \\ \bm{g},\bm{h}\in A^{r-1}, \bm{g}\neq 0}} \sum_{i+j=a+b} \frac{1}{\left\langle \bm{g},\mathbf{z}'\right\rangle^j}\left( \frac{(-1)^b\binom{j-1}{b-1}}{(fz_1+\left\langle \bm{h},\mathbf{z}'\right\rangle)^i}+\frac{(-1)^{j-a}\binom{j-1}{a-1}}{(fz_1+\left\langle \bm{h}-\bm{g},\mathbf{z}'\right\rangle)^i} \right) \\
        ={} & G_r^{=d}(a+b;\mathbf{z})
        + \sum_{\substack{i+j=a+b \\ q-1\mid  j}} \sum_{\substack{f\in A_{+,d} \\ \bm{g},\bm{h}\in A^{r-1},\langle \bm{g},\mathbf{z}'\rangle \succ 0}} \frac{1}{\left\langle \bm{g}, \bfz' \right\rangle^j}\frac{(-1)^{b-1}\binom{j-1}{b-1}+(-1)^{j-a-1}\binom{j-1}{a-1}}{(fz_1+\left\langle \bm{h},\mathbf{z}'\right\rangle)^i} \\
        ={} & G_r^{=d}(a+b;\mathbf{z})
        + \sum_{\substack{i+j=a+b \\ q-1\mid  j}} \Delta^{i,j}_{a,b}E_{r-1}(j;\mathbf{z}')G_r^{=d}(i;\mathbf{z})  \quad\text{(by \eqref{eq-goss-power-sum-analog})}.
    \end{align*}
    For the second summand, we note that it vanishes if $d=0$.
    For $d>0$, a similar argument shows
    \begin{align*}
        & \sum_{\substack{f, g\in A_{+,d} \\ f\neq g}}
        G_a^{\Lambda_{\mathbf{z}'}}(t_{\Lambda_{\mathbf{z}'}}(f z_1))G_b^{\Lambda_{\mathbf{z}'}}(t_{\Lambda_{\mathbf{z}'}}(gz_1))
        = \sum_{\substack{f,g\in A_{+,d} \\ f\neq g \\ \bm{f},\bm{g}\in A^{r-1}}} 
        \frac{1}{(fz_1 + \left\langle \bm{f}, \mathbf{z}' \right\rangle)^a (gz_1+\left\langle \bm{g}, \mathbf{z}'\right\rangle)^b} \quad\text{(by \eqref{eq-Goss-polynomials})} \\
        ={} & \sum_{\substack{f,g\in A_{+,d} \\ f\neq g\\ \bm{f},\bm{g}\in A^{r-1}}}~\sum_{i+j=a+b}\frac{1}{((f-g)z_1+\left\langle \bm{f}-\bm{g},\mathbf{z}'\right\rangle)^j}
        \left( \frac{(-1)^b\binom{j-1}{b-1}}{(fz_1+\left\langle \bm{f},\mathbf{z}'\right\rangle)^i}+\frac{(-1)^{j-a}\binom{j-1}{a-1}}{(gz_1 + \left\langle \bm{g},\mathbf{z}'\right\rangle)^i} \right) \quad\text{(by \eqref{eq-partial-frac})} \\
        ={} & \sum_{i+j=a+b} \left( \sum_{\substack{f\in A_{+,d} \\ h\in A_{<d} \\ \bm{f},\bm{h}\in A^{r-1}}} \frac{(-1)^b\binom{j-1}{b-1}}{(hz_1+\left\langle \bm{h},\mathbf{z}'\right\rangle)^j(fz_1+\left\langle \bm{f},\mathbf{z}'\right\rangle)^i}
        + \sum_{\substack{g\in A_{+,d} \\ h\in A_{<d} \\ \bm{g},\bm{h}\in A^{r-1}}} \frac{(-1)^{j-a}\binom{j-1}{a-1}}{(hz_1+ \left\langle \bm{h},\mathbf{z}'\right\rangle)^j(gz_1+\left\langle \bm{g},\mathbf{z}'\right\rangle)^i} \right)\\
        ={} & \sum_{i+j=a+b}~\sum_{\substack{f,g\in A_+ \\ d=\deg f>\deg g \\ \bm{f},\bm{g}\in A^{r-1} }}~\sum_{\epsilon\in\Fq^\times}\frac{1}{\epsilon^j}\frac{(-1)^b\binom{j-1}{b-1}+(-1)^{j-a}\binom{j-1}{a-1}}{(fz_1+\left\langle \bm{f},\mathbf{z}'\right\rangle)^i(gz_1+ \left\langle \bm{g},\mathbf{z}'\right\rangle)^j} =\sum_{\substack{i+j=a+b \\ q-1\mid  j}}\Delta^{i,j}_{a,b}G_r^{=d}(i,j)  \quad\text{(by \eqref{eq-goss-power-sum-analog})}.
    \end{align*}
    This completes the proof.
\end{proof}

\section{Proof of the Main Theorem}

\subsection{The \texorpdfstring{$q$}{q}-shuffle Algebra \texorpdfstring{$\Ecal$}{E} and the Map \texorpdfstring{$\bfG_r$}{G\_ r}}

In view of Proposition \ref{prop-expansion-MES}, to establish the $q$-shuffle relations for the multiple Eisenstein series, we should define a new $q$-shuffle algebra to deal with the $q$-shuffle relations for multiple Goss sums.

\begin{df}\label{df-Ecal}
    Let $\Tcal$ be the free monoid generated by the set $\{x_k, y_\ell \mid  k, \ell\in \NN\}$, subject to the commutativity relations $x_k y_\ell = y_\ell x_k$ for all $k,\ell\in \NN$,
    and let $\Ecal$ be the $\Fp$-vector space generated by $\Tcal$.
    For a non-empty index $\ww{a} = (a_1,\ldots, a_m)\in\NN^m$, we define $x_{\ww{a}} := x_{a_1}\cdots x_{a_m}$ and $y_{\ww{a}} := y_{a_1}\cdots y_{a_m}$. For $\ww{a}=\emptyset$, we define $x_{\emptyset} = y_{\emptyset} := 1$.
    An element $x_{\ww{a}}$ or $y_{\ww{a}}$ is called a \textit{word of depth $m$}.
    We define the \textit{$q$-shuffle product} $\ast$ on $\Ecal$ inductively on the sum of depths as follows:
    \begin{enumerate}
        \item For the empty word $x_\emptyset = y_\emptyset = 1$ and any $\afk \in \Ecal$, define 
        \begin{equation}  \label{first-relation} 
            1 \ast \afk=\afk \ast 1 = \afk.
        \end{equation}
        
        \item For any indices $\ww{a}$ and $\ww{b}$, define 
        \begin{equation} \label{second-relation}
            x_{\ww{a}} \ast y_{\ww{b}}
        = x_{\ww{a}} y_{\ww{b}}
        = y_{\ww{b}} x_{\ww{a}} 
        = y_{\ww{b}} \ast x_{\ww{a}}.
        \end{equation}
        
        \item For non-empty indices $\ww{a} = (a_1,\ldots, a_m) \in \NN^m$ and $\ww{b} = (b_1,\ldots, b_n)\in \NN^n$, define
        \begin{multline} \label{third-relation}
            x_{\ww{a}} \ast x_{\ww{b}}
            = x_{a_1} ( x_{\ww{a}^{(1)}} \ast x_{\ww{b}} ) 
            + x_{b_1} ( x_{\ww{a}} \ast x_{\ww{b}^{(1)}} ) 
            + x_{a_1 + b_1} ( x_{\ww{a}^{(1)}} \ast x_{\ww{b}^{(1)}} ) \\
            + \sum_{\substack{i+j = a_1 + b_1 \\ q-1 \mid  j}} 
            \Delta^{i,j}_{a_1, b_1} x_i ( ( x_{\ww{a}^{(1)}} \ast x_{\ww{b}^{(1)}} ) * x_j ).
        \end{multline}
        
        \item For non-empty indices $\ww{a} = (a_1,\ldots, a_m) \in \NN^m$ and $\ww{b} = (b_1,\ldots, b_n)\in \NN^n$, define  
        \begin{multline} \label{fourth-relation}
            y_{\ww{a}} \ast y_{\ww{b}}
            = y_{a_1} ( y_{\ww{a}^{(1)}} \ast y_{\ww{b}} ) 
            + y_{b_1} ( y_{\ww{a}} \ast y_{\ww{b}^{(1)}} ) 
            + y_{a_1+b_1} ( y_{\ww{a}^{(1)}} \ast y_{\ww{b}^{(1)}} ) \\
            + \sum_{\substack{i+j = a_1 + b_1 \\ q-1 \mid  j}} 
            \Delta^{i,j}_{a_1,b_1} y_i ( ( y_{\ww{a}^{(1)}} \ast y_{\ww{b}^{(1)}} ) \ast x_j )
            + \sum_{\substack{i+j = a_1 + b_1 \\ q-1 \mid  j}} 
            \Delta^{i,j}_{a_1,b_1} y_i ( ( y_{\ww{a}^{(1)}} \ast y_{\ww{b}^{(1)}} ) * y_j ).
        \end{multline}
        
        \item For any indices $\ww{a},\ww{a}',\ww{b},\ww{b}'$, define
        \begin{equation} \label{fifth-relation}
            (x_{\ww{a}} y_{\ww{b}}) \ast (x_{\ww{a}'}y_{\ww{b}'}) 
        = (x_{\ww{a}}*x_{\ww{a}'}) \ast (y_{\ww{b}}*y_{\ww{b}'}).
        \end{equation}
        \item Expand the product $\ast$ to the $\Fp$-vector space $\Ecal$ by the distributive law.
    \end{enumerate}
\end{df}

One sees that $\Ecal$ is a commutative $\F_p$-algebra, which contains $\Rcal$ as a subalgebra (recall Definition \ref{df-zeta-values-as-words}).

\begin{df}\label{df-bfG} 
    For each $d\in \ZZ_{\ge 0}$, we define $\bfG_{r}^{=d}:\Ecal\to \mathcal{O}(\Omega^r)$ to be the unique $\mathbb{F}_p$-linear map such that for any indices $\ww{a}$ and $\ww{b}$,
    \[
    \bfG_r^{=d}(x_{\ww{a}}y_{\ww{b}}) := E_{r-1}(\ww{a}; \bfz')G_r^{=d}(\ww{b}; \bfz).
    \]
    Here, we recall that $E_1(\ww{a};\bfz')=\zeta_A(\ww{a})$ and $E_{r-1}(\emptyset; \bfz)= 1$ by convention.
    We further define for $d\in \NN$ and $\afk\in \Ecal$,
    \[
    \bfG_r^{<d}(\afk) := \sum_{i=0}^{d-1}\bfG_r^{=i}(\afk) \quad \text{and}\quad
    \bfG_r(\afk) := \lim_{d\to \infty}\bfG_r^{<d}(\afk).
    \]
\end{df}

Note that for any non-empty index $\ww{a}$, we have (recall the convention in Definition \ref{df-Goss-sum})
\[
\bfG_r^{=d}(x_{\ww{a}})
=\bfG_r^{=d}(x_{\ww{a}}y_{\emptyset})
=E_{r-1}(\ww{a}; \bfz')G_r^{=d}(\emptyset; \bfz)
=E_{r-1}(\ww{a}; \bfz')\delta_{d,0}
\]
and
\[
\bfG_r^{=d}(y_{\ww{a}})
=\bfG_r^{=d}(x_{\emptyset}y_{\ww{a}})
=E_{r-1}(\emptyset; \bfz')G_r^{=d}(\ww{a}; \bfz)
=G_r^{=d}(\ww{a}; \bfz).
\]
In particular, for $d\in\NN$, we have
\begin{equation}   \label{eq-G-remark}
    \bfG_r^{<d}(x_{\ww{a}})
    =\sum_{i=0}^{d-1} E_{r-1}(\ww{a}; \bfz') \delta_{i,0}
    =E_{r-1}(\ww{a}; \bfz')
    \quad \text{and} \quad
    \bfG_r^{<d}(y_{\ww{a}})
    =\sum_{i=0}^{d-1} G_{r}^{=i}(\ww{a};\bfz)
    =G_r^{<d}(\ww{a};\bfz).
\end{equation}
Moreover, it follows from \eqref{eq-G_d=G_dG_<d} that for any $a \in \NN$ and non-empty index $\ww{b}$, we have
\begin{equation}   \label{eq-G_d=G_dG_<d-bf-version}
    \bfG_r^{=d}(y_{a}y_{\ww{b}})
    =\bfG_r^{=d}(y_{a})\bfG_r^{<d}(y_{\ww{b}}).
\end{equation}

Recall that $\bfE_r$ is defined in Definition \ref{df-E-hat}.
\begin{prop}\label{prop-bfG-ring-hom}
    For $r\ge 2$, we have that if  $\bfE_{r-1}$ is an $\F_p$-algebra homomorphism, then so is $\bfG_r^{<d}$ for each $d\in\NN$.
    Consequently, $\bfG_r$ is also an $\mathbb{F}_p$-algebra homomorphism.
\end{prop}
\begin{proof}
    It suffices to prove that $\bfG_r^{<d}:\mathcal{E}\to \mathcal{O}(\Omega^r)$ is compatible with the $q$-shuffle relations in Definition \ref{df-Ecal}.
    \begin{enumerate}
        \item For relation \eqref{first-relation}, note that
        \[
        \bfG_r^{<d}(1) = \bfG_r^{<d}(x_{\emptyset}) \overset{\eqref{eq-G-remark}}{ = } E_{r-1}(\emptyset; \bfz') = 1.
        \]
        Therefore,
        \[
        \bfG_r^{<d}(1*\afk) = \bfG_r^{<d}(\afk*1) = \bfG_r^{<d}(\afk) = \bfG_r^{<d}(1)\bfG_r^{<d}(\afk).
        \]
        
        \item \label{proof-ii} For relation \eqref{second-relation}, we have
        \begin{align*}
            \bfG_r^{<d}(x_{\ww{a}}*y_{\ww{b}})
            &= \sum_{i=0}^{d-1} \bfG_{r}^{=i}(x_{\ww{a}}y_{\ww{b}})
            = \sum_{i=0}^{d-1} E_{r-1}(\ww{a}; \bfz') G_{r}^{=i}(\ww{b}; \bfz) \\
            &= E_{r-1}(\ww{a}; \bfz')\bfG_r^{<d}(y_{\ww{b}})
            \overset{\eqref{eq-G-remark}}{=} \bfG_r^{<d}(x_{\ww{a}}) \bfG_r^{<d}(y_{\ww{b}}).
        \end{align*}
        
        \item \label{proof-iii} For relation \eqref{third-relation}, we write
        \[
        x_{\ww{a}}*x_{\ww{b}}=\sum_{i}\epsilon_ix_{\ww{v}_i}
        \]
         for some $\epsilon_i \in \FF_p$ and indices $\ww{v}_i$.
        By assumption, as functions on $\Omega^{r-1}$,
        \[
        E_{r-1}(\ww{a};\mathbf{w})E_{r-1}(\ww{b};\mathbf{w})
        = \bfE_{r-1}(x_{\ww{a}}*x_{\ww{b}})
        = \sum_{i}\epsilon_i\bfE_{r-1}(x_{\ww{v}_i})
        =\sum_{i}\epsilon_iE_{r-1}(\ww{v}_i;\mathbf{w}).
        \]
        Therefore,
        \begin{align*}
            \bfG_r^{<d}(x_{\ww{a}}*x_{\ww{b}})
            &= \sum_{i}\epsilon_i\bfG_r^{<d}(x_{\ww{v}_i})
            \overset{\eqref{eq-G-remark}}{=} \sum_{i}\epsilon_iE_{r-1}(\ww{v}_i;\mathbf{z}') \\
            &=E_{r-1}(\ww{a};\mathbf{z}')E_{r-1}(\ww{b};\mathbf{z}')
            \overset{\eqref{eq-G-remark}}{=} \bfG_r^{<d}(x_{\ww{a}})\bfG_r^{<d}(x_{\ww{b}}).
        \end{align*}
        
        \item \label{proof-iv} For relation \eqref{fourth-relation}, namely, 
        $\bfG_r^{<d}(y_{\ww{a}} * y_{\ww{b}}) = \bfG_r^{<d}(y_{\ww{a}}) \bfG_r^{<d}(y_{\ww{b}})$, 
        note that the result reduces to the first relation when either $\dep(\ww{a})$ or $\dep(\ww{b})$ is zero. Assuming both $\ww{a}$ and $\ww{b}$ are non-empty, we proceed by induction on the total depth $\dep(\ww{a}) + \dep(\ww{b})$.
        For the base case $\dep(\ww{a}) = \dep(\ww{b}) = 1$, we write $\ww{a}=a$ and $\ww{b}=b$.
        Then we have
        \begin{align*}
            &\bfG_r^{<d}(y_a)\bfG_r^{<d}(y_b)
            = \left(\sum_{d>d_1}\bfG^{=d_1}_{r}(y_a)\right)\left(\sum_{d>d_2}\bfG^{=d_2}_{r}(y_b)\right)  \\
            ={} &\sum_{d>d_1>d_2} G^{=d_1}_{r}(a;\mathbf{z})G^{=d_2}_{r}(b;\mathbf{z})
            + \sum_{d>d_2>d_1} G_r^{=d_1}(a;\mathbf{z})G_r^{=d_2}(b;\mathbf{z})
            + \sum_{d>d_1=d_2} G_r^{=d_1}(a;\mathbf{z})G_r^{=d_2}(b;\mathbf{z})  \\
            ={} &\sum_{d>d_1}G_r^{=d_1}(a;\mathbf{z})G_r^{<d_1}(b;\mathbf{z})+\sum_{d>d_2}G_r^{=d_2}(b;\mathbf{z})G_r^{<d_2}(a;\mathbf{z})+\sum_{d>d_3}G_r^{=d_3}(a;\mathbf{z})G_r^{=d_3}(b;\mathbf{z}) \\
            ={} &\begin{multlined}[t]
                \sum_{d>d_1}G^{=d_1}_{r}((a,b);\mathbf{z})+\sum_{d>d_2}G^{<d_2}_{r}((b,a);\mathbf{z}) \\
                +\sum_{d>d_3}\Bigg(G_r^{=d_3}(a+b;\mathbf{z}) +\sum_{\substack{i+j = a+b \\q-1\mid j}}\Delta^{i,j}_{a,b}G_r^{=d_3}(i;\bfz)E_{r-1}(j;\bfz')
                +\sum_{\substack{i+j = a+b \\ q-1\mid  j}}\Delta^{i,j}_{a,b}G_r^{=d_3}((i,j);\mathbf{z}) \Bigg) \\ 
                \text{(by \eqref{eq-G_d=G_dG_<d} and Proposition \ref{prop-G_d(r)G_d(s)})}
            \end{multlined} \\
            ={} &\begin{multlined}[t]
                \bfG_r^{<d}(y_a y_b)+\bfG_r^{<d}(y_b y_a)+\bfG_r^{<d}(y_{a+b})  \\
                + \sum_{\substack{i+j = a+b\\q-1\mid  j}}\Delta^{i,j}_{a,b}\bfG_r^{<d}(y_i x_j)+\sum_{\substack{i+j = a+b\\q-1\mid  j}}\Delta^{i,j}_{a,b}\bfG_r^{<d}(y_i y_j)   \quad\text{(by \eqref{eq-G-remark})}
            \end{multlined}  \\
            ={} &\bfG_r^{<d}(y_a * y_b) \quad\text{(by \eqref{fourth-relation})}.
        \end{align*}
        Thus, the base case is proved.

        Given $N>2$, suppose the claim holds for any indices with total depth $\le N-1$.
        Let $\ww{a}$ and $\ww{b}$ be non-empty indices with $\dep(\ww{a})+\dep(\ww{b})=N$. Then
        \begin{align*}
            &\bfG_r^{<d}(y_{\ww{a}})\bfG_r^{<d}(y_{\ww{b}})
            =\left(\sum_{d>d_1}\bfG_r^{=d_1}(y_{\ww{a}})\right)\left(\sum_{d>d_2}\bfG_r^{=d_2}(y_{\ww{b}})\right) \\
            ={} &\sum_{d>d_1>d_2}G_r^{=d_1}(\ww{a};\bfz)G_r^{=d_2}(\ww{b};\bfz)+\sum_{d>d_2>d_1}G_r^{=d_1}(\ww{a};\bfz)G_r^{=d_2}(\ww{b};\bfz)+\sum_{d>d_1=d_2}G_r^{=d_1}(\ww{a};\bfz)G_r^{=d_2}(\ww{b};\bfz)  \\
            ={} &\begin{multlined}[t]
                \sum_{d>d_1}G_r^{=d_1}(a_1;\bfz)G_r^{<d_1}(\ww{a}^{(1)};\bfz)G_r^{<d_1}(\ww{b};\bfz)
                + \sum_{d>d_2}G_r^{=d_2}(b_1;\bfz)G_r^{<d_2}(\ww{a};\bfz)G_r^{<d_2}(\ww{b}^{(1)};\bfz)  \\
                + \sum_{d>d_3}G_r^{=d_3}(a_1;\bfz)G_r^{=d_3}(b_1;\bfz)G_r^{<d_3}(\ww{a}^{(1)};\bfz)G_r^{<d_3}(\ww{b}^{(1)};\bfz) \quad\text{(by \eqref{eq-G_d=G_dG_<d})}.
            \end{multlined} 
        \end{align*}
        For the first summand, we have
        \begin{alignat*}{2}
            &\sum_{d>d_1}G_r^{=d_1}(a_1;\mathbf{z})G_r^{<d_1}(\ww{a}^{(1)};\mathbf{z})G_r^{<d_1}(\ww{b};\mathbf{z}) \\
            ={} &\sum_{d>d_1}\bfG_r^{=d_1}(y_{a_1})\bfG_r^{<d_1}(y_{\ww{a}^{(1)}}*y_{\ww{b}}) \quad &&\text{(by \eqref{eq-G-remark} and induction hypothesis)} \\
            ={} &\sum_{d>d_1}\bfG_r^{=d_1} \left(y_{a_1}(y_{\ww{a}^{(1)}}*y_{\ww{b}})\right) \quad &&\text{(by \eqref{eq-G_d=G_dG_<d-bf-version})}  \\
            ={} &\bfG_r^{<d}\left(y_{a_1}(y_{\ww{a}^{(1)}}*y_{\ww{b}})\right).
        \end{alignat*}
        Similarly, for the second summand,
        \begin{align*}
            \sum_{d>d_2}G_r^{=d_2}(b_1;\mathbf{z})G_r^{<d_2}(\ww{a};\mathbf{z})G_{r}^{<d_2}(\ww{b}^{(1)};\mathbf{z})
            &=\bfG_r^{<d} \left( y_{b_1}(y_{\ww{a}}*y_{\ww{b}^{(1)}}) \right).
        \end{align*}
        For the third summand, we have
        \begin{align*}
            &\sum_{d>d_3}G_r^{=d_3}(a_1;\mathbf{z})G_r^{=d_3}(b_1;\mathbf{z})G_r^{<d_3}(\ww{a}^{(1)};\mathbf{z})G_r^{<d_3}(\ww{b}^{(1)};\mathbf{z})\\
            ={} &\begin{multlined}[t]
                \sum_{d>d_3} \Bigg( G^{=d_3}_{r}(a_1+b_1;\mathbf{z})G_r^{<d_3}(\ww{a}^{(1)};\mathbf{z})G_r^{<d_3}(\ww{b}^{(1)};\mathbf{z}) \\
                + \sum_{\substack{i+j = a_1+b_1\\q-1\mid  j}}\Delta^{i,j}_{a_1,b_1}G_r^{=d_3}(i;\mathbf{z})G_r^{<d_3} (\ww{a}^{(1)};\mathbf{z})G_r^{<d_3}(\ww{b}^{(1)};\mathbf{z})E_{r-1}(j;\bfz') \\
                + \sum_{\substack{i+j=a_1+b_1\\q-1\mid  j}}\Delta^{i,j}_{a_1,b_1} G_r^{=d_3}(i;\mathbf{z})G^{<d_3}_{r}(\ww{a}^{(1)};\mathbf{z}) G_r^{<d_3}(\ww{b}^{(1)};\mathbf{z})G_r^{<d_3}(j;\mathbf{z}) \Bigg) \\
                \text{(by Proposition \ref{prop-G_d(r)G_d(s)} and \eqref{eq-G_d=G_dG_<d})}
            \end{multlined}\\
            ={} &\begin{multlined}[t]
                \sum_{d>d_3} \Bigg( \bfG_r^{=d_3}(y_{a_1+b_1})\bfG_r^{<d_3}(y_{\ww{a}^{(1)}} * y_{\ww{b}^{(1)}})
                + \sum_{\substack{i+j = a_1+b_1\\q-1\mid  j}}\Delta^{i,j}_{a_1,b_1}\bfG_r^{=d_3}(y_i)\bfG_r^{<d_3}\left( (y_{\ww{a}^{(1)}} * y_{\ww{b}^{(1)}})*x_j \right) \\
                + \sum_{\substack{i+j = a_1+b_1\\q-1\mid  j}}\Delta^{i,j}_{a_1,b_1}\bfG_r^{=d_3}(y_i)\bfG_r^{<d_3}\left( (y_{\ww{a}^{(1)}} * y_{\ww{b}^{(1)}}) *y_j \right) \Bigg)
                \quad \text{(by \eqref{eq-G-remark} and induction hypothesis)}
            \end{multlined} \\
            ={} &\begin{multlined}[t]
                \sum_{d>d_3} \Bigg( \bfG_r^{= d_3 }\left(y_{a_1+b_1}(y_{\ww{a}^{(1)}} * y_{\ww{b}^{(1)}}) \right)
                + \sum_{\substack{i+j = a_1+b_1\\q-1\mid  j}}\Delta^{i,j}_{a_1,b_1}\bfG_r^{= d_3}\left( y_i \left((y_{\ww{a}^{(1)}} * y_{\ww{b}^{(1)}})*x_j \right)\right) \\
                + \sum_{\substack{i+j = a_1+b_1\\q-1\mid  j}}\Delta^{i,j}_{a_1,b_1}\bfG_r^{= d_3}\left( y_i \left((y_{\ww{a}^{(1)}} * y_{\ww{b}^{(1)}}) *y_j \right)\right) \Bigg)
                \quad \text{(by \eqref{eq-G_d=G_dG_<d-bf-version})}
            \end{multlined} \\
            ={} &\begin{multlined}[t]
                \bfG_r^{< d}\left(y_{a_1+b_1}(y_{\ww{a}^{(1)}} * y_{\ww{b}^{(1)}}) \right)
                + \sum_{\substack{i+j = a_1+b_1\\q-1\mid  j}}\Delta^{i,j}_{a_1,b_1}\bfG_r^{< d}\left( y_i \left( (y_{\ww{a}^{(1)}} * y_{\ww{b}^{(1)}})*x_j \right) \right) \\
                + \sum_{\substack{i+j = a_1+b_1\\q-1\mid  j}}\Delta^{i,j}_{a_1,b_1}\bfG_r^{< d}\left( y_i\left( (y_{\ww{a}^{(1)}} * y_{\ww{b}^{(1)}}) *y_j \right) \right).
            \end{multlined}
        \end{align*}
        
        Combining these three summands, we have
        \begin{align*}
            &\bfG_r^{<d}(y_{\ww{a}})\bfG_r^{<d}(y_{\ww{b}}) \\
            ={} &\begin{multlined}[t]
                \bfG_r^{<d} \left(y_{a_1}(y_{\ww{a}^{(1)}}*y_{\ww{b}}) \right)
                + \bfG_r^{<d} \left(y_{b_1}(y_{\ww{a}}*y_{\ww{b}^{(1)}}) \right)
                + \bfG_r^{<d} \left(y_{a_1+b_1}(y_{\ww{a}^{(1)}}*y_{\ww{b}^{(1)}})\right) \\
                + \sum_{\substack{i+j = a_1+b_1\\q-1\mid  j}} \Delta^{i,j}_{a_1,b_1} \bfG_r^{<d}\left(y_i\left( (y_{\ww{a}^{(1)}}*y_{\ww{b}^{(1)}})*x_j \right) \right)
                + \sum_{\substack{i+j = a_1+b_1\\q-1\mid  j}} \Delta^{i,j}_{a_1,b_1} \bfG_r^{<d}\left(y_i\left( (y_{\ww{a}^{(1)}}*y_{\ww{b}^{(1)}})*y_j \right) \right)
            \end{multlined} \\
            ={} &\bfG_r^{<d}(y_{\ww{a}}*y_{\ww{b}}) \quad\text{(by \eqref{fourth-relation})}.
        \end{align*}
        Therefore, the result holds by induction.
        
        \item For relation \eqref{fifth-relation}, namely, $\bfG_r^{<d}((x_{\ww{a}} y_{\ww{b}}) \ast (x_{\ww{a}'}y_{\ww{b}'}))
        = \bfG_r^{<d}(x_{\ww{a}} y_{\ww{b}}) 
        \bfG_r^{<d}(x_{\ww{a}'} y_{\ww{b}'})$, first note that since $x_{\ww{a}} \ast x_{\ww{a'}} \in \Rcal$, we have
        \begin{alignat*}{2}
            \bfG_r^{<d} \left( (x_{\ww{a}} \ast (x_{\ww{a}'} y_{\ww{b}}) \right) 
            &= \bfG_r^{<d}((x_{\ww{a}} * x_{\ww{a}'}) * y_{\ww{b}}) \quad &&\text{(by \eqref{fifth-relation})}  \\
            &= \bfG_r^{<d}(x_{\ww{a}} * x_{\ww{a}'}) \bfG_r^{<d}(y_{\ww{b}}) \quad &&\text{(by \ref{proof-ii})}  \\
            &= \bfG_r^{<d}(x_{\ww{a}}) \bfG_r^{<d}(x_{\ww{a}'}) \bfG_r^{<d}(y_{\ww{b}}) \quad &&\text{(by \ref{proof-iii})} \\
            &= \bfG_r^{<d}(x_{\ww{a}}) \bfG_r^{<d}(x_{\ww{a}'}  y_{\ww{b}}) \quad &&\text{(by \ref{proof-ii} and \eqref{second-relation})}.
        \end{alignat*}
        Now, we write
        \[
        y_{\ww{b}}*y_{\ww{b}'}
        =\sum_i \epsilon_i x_{\ww{v}_i} y_{\ww{w}_i}
        \]
        for some $\epsilon_i \in \Fp$ and indices $\ww{v}_i,\ww{w}_i$. 
        Then we obtain
        \begin{alignat*}{2}
            &\bfG_r^{<d} \left( (x_{\ww{a}}y_{\ww{b}}) *(x_{\ww{a}'}y_{\ww{b}'}) \right)  \\
            ={} &\bfG_r^{<d} \left( (x_{\ww{a}} * x_{\ww{a}'}) * (y_{\ww{b}} * y_{\ww{b}'}) \right) 
            \quad &&\text{(by \eqref{fifth-relation})} \\
            ={} &\sum_{i} \epsilon_i \bfG_r^{<d} \left( (x_{\ww{a}} * x_{\ww{a}'}) * ( x_{\ww{v}_i} y_{\ww{w}_i} ) \right)  \\
            ={} &\sum_{i} \epsilon_i \bfG_r^{<d}(x_{\ww{a}} * x_{\ww{a}'}) \bfG_r^{<d}(x_{\ww{v}_i}y_{\ww{w}_i}) \quad &&\text{(by the argument above)}   \\
            ={} &\bfG_r^{<d}(x_{\ww{a}} * x_{\ww{a}'})\bfG_r^{<d}(y_{\ww{b}} * y_{\ww{b}'}) \\
            ={} &\bfG_r^{<d}(x_{\ww{a}}) \bfG_r^{<d}(x_{\ww{a}'}) \bfG_r^{<d}(y_{\ww{b}})  \bfG_r^{<d}(y_{\ww{b}'}) 
            \quad &&\text{(by \ref{proof-iii} and \ref{proof-iv})} \\
            ={} &\bfG_r^{<d}(x_{\ww{a}}y_{\ww{b}})\bfG_r^{<d}(x_{\ww{a}'}y_{\ww{b}'})
            \quad &&\text{(by \ref{proof-ii})}.
        \end{alignat*}
    \end{enumerate}
\end{proof}

\subsection{The Map \texorpdfstring{$\bfe$}{e}}

We let $\mathfrak{A}$ be the ideal of $\Ecal$ generated by the set
\[
\{(\afk*\bfk)*\cfk-\afk*(\bfk*\cfk) \mid  \afk,\bfk,\cfk\in \mathcal{E}\},
\]
so that $\mathcal{E}/\mathfrak{A}$ is a commutative and associative algebra.

\begin{rem} \label{rmk-ass}
    For $q < 10$ and words $\afk, \bfk, \cfk$ with total weight less than $8$, we have checked the associativity of $\Ecal$ using \texttt{SageMath}.
    Based on these calculations, we conjecture that $\Ecal$ is associative, which will be explored in a future project.
\end{rem}

\begin{df}\label{df-bfe}
    We define $\bfe:\Rcal \to \Ecal/\mathfrak{A}$ to be the unique $\F_p$-linear map such that for any non-empty index $\ww{a}$, 
    \[
    \bfe(1) := 1 \pmod{\mathfrak{A}}
    \quad
    \text{and}
    \quad
    \bfe(x_{\ww{a}}) := x_{\ww{a}}+\sum_{i=1}^{\dep(\ww{a})}x_{\ww{a}^{(i)}}y_{\ww{a}_{(i)}} \pmod{\mathfrak{A}}.
    \]
\end{df}

Our goal is to show that $\bfe$ is an $\Fp$-algebra homomorphism, which is stated in Proposition \ref{prop-bfe-alg-hom}.
We first establish the following lemma:

\begin{lem}\label{lem-(yb)*a=y(b*a)}
    For each $\afk\in \mathcal{R}$, $\bfk\in \mathcal{E}$ and $w\in\bbN$, we have 
    \[
    y_w(\bfk*\afk)=(y_w\bfk)*\afk.
    \]
\end{lem}

\begin{proof}
    By the distributive law, it suffices to show that for all $w\in \mathbb{N}$ and indices $\ww{a}, \ww{b},\ww{c}$, we have
    \[
    y_w\left((y_{\ww{a}}x_{\ww{b}}\right) * x_{\ww{c}})=(y_wy_{\ww{a}}x_{\ww{b}})*x_{\ww{c}}.
    \]
    Write 
    \[
    x_{\ww{b}}*x_{\ww{c}}=\sum_{i}\epsilon_ix_{\ww{v}_i}
    \]
    for some $\epsilon_i\in \mathbb{F}_p$ and indices $\ww{v}_i$.
    Then we have
    \begin{align*}
        y_w \left((y_{\ww{a}}x_{\ww{b}}) * x_{\ww{c}} \right)
        &= y_w \left( y_{\ww{a}}*(x_{\ww{b}}*x_{\ww{c}}) \right)
        = y_w \left(y_{\ww{a}}\sum_{i}a_ix_{\ww{v}_i}\right)\\
        &=y_wy_{\ww{a}}\left(\sum_{i}a_ix_{\ww{v}_i}\right)=(y_wy_{\ww{a}})*(x_{\ww{b}}*x_{\ww{c}})
        =(y_w y_{\ww{a}}x_{\ww{b}})*x_{\ww{c}}.
    \end{align*}
\end{proof}

\begin{prop}\label{prop-bfe-alg-hom}
    $\bfe:\mathcal{R}\to \mathcal{E}/\mathfrak{A}$ is an $\mathbb{F}_p$-algebra homomorphism.
\end{prop}
\begin{proof}
    We first observe that for all $a\in \mathbb{N}$ and non-empty index $\ww{a}$,
    \begin{equation} \label{eq-bfe-alg-hom-observation}
        \bfe(x_ax_{\ww{a}})=x_ax_{\ww{a}}+y_ax_{\ww{a}}+\sum_{i=1}^{\dep(\ww{a})}y_ax_{\ww{a}^{(i)}}y_{\ww{a}_{(i)}}=x_ax_{\ww{a}}+y_a\bfe(x_{\ww{a}}).
    \end{equation}
    We need to prove
    \[
    \bfe(x_{\ww{a}})*\bfe(x_{\ww{b}})\equiv \bfe(x_{\ww{a}}*x_{\ww{b}})\pmod{\mathfrak{A}}
    \]
    for any indices $\ww{a},\ww{b}$.
    If either $\dep(\ww{a})$ or $\dep(\ww{b})$ is zero, then $\bfe(x_{\ww{a}})=1$ or $\bfe(x_{\ww{b}})=1$, and the result follows.
    Assuming both $\ww{a}$ and $\ww{b}$ are non-empty, we proceed by induction on the total depth $\dep(\ww{a})+\dep(\ww{b})$.
    For the base case $\dep({\ww{a}}) = \dep(\ww{b}) = 1$, one computes that

    \begin{align*}
        \bfe(x_a*x_b)
        &=\begin{multlined}[t]
            x_a x_b+x_b x_a+x_{a+b}+\sum_{\substack{i+j=a+b \\ q-1\mid j}}\Delta^{i,j}_{a,b}x_ix_j+y_ax_b+y_bx_a+y_{a+b} \\
            +\sum_{\substack{i+j=a+b \\ q-1\mid j}}\Delta^{i,j}_{a,b}y_ix_j+y_ay_b+y_by_a+\sum_{\substack{i+j=a+b \\ q-1\mid j}}\Delta^{i,j}_{a,b}y_iy_j \\
        \end{multlined} \\
        &=(x_a*x_b)+(y_b*x_a)+(x_a*y_b)+(y_a*y_b) \\
        &=(x_a+y_a)*(x_b+y_b)  \\
        &=\bfe(x_a)*\bfe(x_b).
    \end{align*}
    
    Given $N>2$, suppose the claim holds for any indices with total depth less than $N$.
    Let $\ww{a}$ and $\ww{b}$ be non-empty indices with $\dep(\ww{a})+\dep(\ww{b})=N$.
    Then
    \begin{align*}
        &\bfe(x_{\ww{a}}) * \bfe(x_{\ww{b}}) \\
        ={} &\left( x_{\ww{a}}+y_{a_1} \bfe(x_{\ww{a}^{(1)}}) \right) * \left(x_{\ww{b}}+y_{b_1}\bfe(x_{\ww{b}^{(1)}}) \right) \quad \text{(by \eqref{eq-bfe-alg-hom-observation})}\\
        ={} &x_{\ww{a}}*x_{\ww{b}}
        +\left( y_{a_1}\bfe(x_{\ww{a}^{(1)}}) \right) * x_{\ww{b}}
        +\left(y_{\ww{b}_1}\bfe(x_{\ww{b}^{(1)}})\right)*x_{\ww{a}}
        +\left( y_{a_1}\bfe(x_{\ww{a}^{(1)}}) \right) * \left(y_{b_1} \bfe(x_{\ww{b}^{(1)}}) \right).
    \end{align*}
    On the other hand,
    \begin{align*}
        &\bfe(x_{\ww{a}} * x_{\ww{b}}) \\
        ={} &\begin{multlined}[t]
            \bfe\left( x_{a_1}(x_{\ww{a}^{(1)}} * x_{\ww{b}}) \right)+\bfe\left( x_{b_1}(x_{\ww{a}}*x_{\ww{b}^{(1)}}) \right)
            +\bfe\left( x_{a_1+b_1}(x_{\ww{a}^{(1)}}*x_{\ww{b}^{(1)}}) \right) \\
            +\sum_{\substack{i+j=a_1+b_1 \\ q-1\mid j}}\Delta^{i,j}_{a_1,b_1}\bfe\left( x_i\left( (x_{\ww{a}^{(1)}}*x_{\ww{b}^{(1)}})*x_j \right) \right) \quad \text{(by \eqref{third-relation})}
        \end{multlined} \\
        ={} &\begin{multlined}[t]
            x_{a_1}(x_{\ww{a}^{(1)}}*x_{\ww{b}})+x_{b_1}(x_{\ww{a}}*x_{\ww{b}^{(1)}})+x_{a_1+b_1}(x_{\ww{a}^{(1)}}*x_{\ww{b}^{(1)}}) \\
            +\sum_{\substack{i+j=a_1+b_1 \\ q-1\mid j}}\Delta^{i,j}_{a_1,b_1}x_i\left( (x_{\ww{a}^{(1)}}*x_{\ww{b}^{(1)}})*x_j \right) \\ +y_{a_1}\bfe(x_{\ww{a}^{(1)}}*x_{\ww{b}})+y_{b_1}\bfe(x_{\ww{a}}*x_{\ww{b}^{(1)}})+y_{a_1+b_1}\bfe(x_{\ww{a}^{(1)}}*x_{\ww{b}^{(1)}}) \\
            +\sum_{\substack{i+j=a_1+b_1 \\ q-1\mid j}}\Delta^{i,j}_{a_1,b_1}y_i\bfe\left( (x_{\ww{a}^{(1)}}*x_{\ww{b}^{(1)}})*x_j \right)
            \quad\text{(by \eqref{eq-bfe-alg-hom-observation})}
        \end{multlined} \\
        ={} &\begin{multlined}[t]
            x_{\ww{a}}*x_{\ww{b}} +y_{a_1}\left(\bfe(x_{\ww{a}^{(1)}})*\bfe(x_{\ww{b}})\right)+y_{b_1}\left(\bfe(x_{\ww{a}})*\bfe(x_{\ww{b}^{(1)}})\right)
            + y_{a_1+b_1}\left(\bfe(x_{\ww{a}^{(1)}})*\bfe(x_{\ww{b}^{(1)}})\right) \\
            +\sum_{\substack{i+j=a_1+b_1 \\ q-1\mid j}}\Delta^{i,j}_{a_1,b_1}y_i\left( \left( \bfe(x_{\ww{a}^{(1)}})*\bfe(x_{\ww{b}^{(1)}})\right) *\bfe(x_j) \right) \\
            \text{(by \eqref{third-relation} and induction hypothesis)}
        \end{multlined} \\
        ={} &\begin{multlined}[t]
            x_{\ww{a}}*x_{\ww{b}}
            +y_{a_1}\left( \bfe(x_{\ww{a}^{(1)}})* \left( x_{\ww{b}}+y_{b_1} \bfe(x_{\ww{b}^{(1)}}) \right) \right)
            +y_{b_1} \left(\bfe(x_{\ww{b}^{(1)}})*\left(x_{\ww{a}}+y_{a_1}\bfe(x_{\ww{a}^{(1)}}) \right)\right) \\
            +y_{a_1+b_1}\bfe(x_{\ww{a}^{(1)}}*x_{\ww{b}^{(1)}})
            +\sum_{\substack{i+j=a_1+b_1 \\ q-1\mid j}}\Delta^{i,j}_{a_1,b_1} y_i\left( \left( \bfe(x_{\ww{a}^{(1)}})*\bfe(x_{\ww{b}^{(1)}})\right) *\bfe(x_j) \right)
            \quad\text{(by \eqref{eq-bfe-alg-hom-observation})}
        \end{multlined} \\
        ={} &\begin{multlined}[t]
            x_{\ww{a}}*x_{\ww{b}}
            +y_{a_1}\left( \bfe(x_{\ww{a}^{(1)}})* x_{\ww{b}}\right)+y_{a_1}\left(\bfe(x_{\ww{a}^{(1)}})*\left(y_{b_1} \bfe(x_{\ww{b}^{(1)}})\right)\right) 
            \\+y_{b_1} \left(\bfe(x_{\ww{b}^{(1)}})*x_{\ww{a}}\right)+y_{b_1}\left(\bfe(x_{\ww{b}^{(1)}})*\left(y_{a_1}\bfe(x_{\ww{a}^{(1)}})\right) \right)
            +y_{a_1+b_1}\bfe(x_{\ww{a}^{(1)}}*x_{\ww{b}^{(1)}}) \\
            +\sum_{\substack{i+j=a_1+b_1 \\ q-1\mid j}}\Delta^{i,j}_{a_1,b_1} y_i\left( \left( \bfe(x_{\ww{a}^{(1)}})*\bfe(x_{\ww{b}^{(1)}})\right) *\bfe(x_j) \right)
        \end{multlined} \\
        ={} &\begin{multlined}[t]
            x_{\ww{a}}*x_{\ww{b}}
            +\left( y_{a_1} \bfe(x_{\ww{a}^{(1)}})\right) * x_{\ww{b}}
            +y_{a_1}\left(\bfe(x_{\ww{a}^{(1)}}) * \left( y_{b_1}\bfe(x_{\ww{b}^{(1)}})\right)\right) \\
            +\left(y_{b_1}\bfe(x_{\ww{b}^{(1)}})\right) * x_{\ww{a}}
            +y_{b_1} \left( \left( y_{a_1}\bfe(x_{\ww{a}^{(1)}}) \right) * \bfe(x_{\ww{b}^{(1)}}) \right)
            +y_{a_1+b_1}\bfe(x_{\ww{a}^{(1)}}*x_{\ww{b}^{(1)}}) \\
            +\sum_{\substack{i+j=a_1+b_1 \\ q-1\mid j}}\Delta^{i,j}_{a_1,b_1} y_i\left( \left( \bfe(x_{\ww{a}^{(1)}})*\bfe(x_{\ww{b}^{(1)}})\right) *\bfe(x_j) \right)
            \quad\text{(by Lemma \ref{lem-(yb)*a=y(b*a)})}.
        \end{multlined}
    \end{align*}
    Therefore, it is reduced to show that
    \begin{align*}
        &\left( y_{a_1}\bfe(x_{\ww{a}^{(1)}})\right) * \left(y_{b_1}\bfe(x_{\ww{b}^{(1)}})\right) \\
        \equiv{} &\begin{multlined}[t]
            y_{a_1}\left(\bfe(x_{\ww{a}^{(1)}}) * \left( y_{b_1}\bfe(x_{\ww{b}^{(1)}})\right)\right)
            +y_{b_1} \left( \left( y_{a_1}\bfe(x_{\ww{a}^{(1)}}) \right) * \bfe(x_{\ww{b}^{(1)}}) \right) \\
            +y_{a_1+b_1}\bfe(x_{\ww{a}^{(1)}}*x_{\ww{b}^{(1)}})
            +\sum_{\substack{i+j=a_1+b_1 \\ q-1\mid j}}\Delta^{i,j}_{a_1,b_1} y_i\left( \left( \bfe(x_{\ww{a}^{(1)}})*\bfe(x_{\ww{b}^{(1)}})\right) *\bfe(x_j) \right)
            \pmod{\mathfrak{A}},
        \end{multlined}
    \end{align*}
    which follows from Proposition \ref{prop-ye(x)*ye(x)} below. 
\end{proof}

Therefore, Proposition \ref{prop-bfe-alg-hom} relies on Proposition \ref{prop-ye(x)*ye(x)} below.
To prove the latter, we begin with the following lemma.

\begin{lem}\label{lem-yx*yx}
    For all indices $\ww{a},\ww{b},\ww{v},\ww{w}$, we have
    \begin{multline*}
        (y_{\ww{a}}x_{\ww{v}})*(y_{\ww{b}}x_{\ww{w}})
        \equiv y_{a_1}\left((y_{\ww{a}^{(1)}}x_{\ww{v}})*(y_{\ww{b}}x_{\ww{w}}) \right)
        + y_{b_1} \left((y_{\ww{a}}x_{\ww{v}})*(y_{\ww{b}^{(1)}}x_{\ww{w}})\right) \\
        + y_{a_1+b_1} \left((y_{\ww{a}^{(1)}}x_{\ww{v}})*(y_{\ww{b}^{(1)}}x_{\ww{w}})\right)
        + \sum_{\substack{i+j=a_1+b_1 \\ q-1 \mid j}} \Delta^{i,j}_{a_1,b_1} y_i \left(\left((y_{\ww{a}^{(1)}}x_{\ww{v}})*(y_{\ww{b}^{(1)}}x_{\ww{w}}) \right)* \bfe(x_j) \right)  \\
        \pmod{\Afk}.
    \end{multline*}
\end{lem}

\begin{proof}
    We observe that the $q$-shuffle product $y_{\ww{a}} \ast y_{\ww{b}}$ can be written as
    \begin{multline}  \label{eq-rewrite-shuffle-of-y}
        y_{\ww{a}} * y_{\ww{b}} = y_{a_1} ( y_{\ww{a}^{(1)}} * y_{\ww{b}} ) 
        + y_{b_1} ( y_{\ww{a}} * y_{\ww{b}^{(1)}} ) 
        + y_{a_1+b_1} ( y_{\ww{a}^{(1)}} * y_{\ww{b}^{(1)}} ) \\
        + \sum_{\substack{i+j = a_1 + b_1 \\ q-1 \mid  j}} 
        \Delta^{i,j}_{a_1,b_1} y_i \left( ( y_{\ww{a}^{(1)}} * y_{\ww{b}^{(1)}} ) * \bfe(x_j) \right).
    \end{multline}
    Set $\mathfrak{a} := x_{\ww{v}}*x_{\ww{w}}$.
    Then we have
    \begin{align*}
        &(y_{\ww{a}}x_{\ww{v}})*(y_{\ww{b}}x_{\ww{w}})
        = (y_{\ww{a}}*y_{\ww{b}})*\mathfrak{a} \\
        ={} &\begin{multlined}[t]
            \left(y_{a_1}(y_{\ww{a}^{(1)}}*y_{\ww{b}})\right) * \mathfrak{a}
            + \left(y_{b_1}(y_{\ww{a}}*y_{\ww{b}^{(1)}})\right) * \mathfrak{a}
            + \left(y_{a_1+b_1}(y_{\ww{a}^{(1)}}*y_{\ww{b}^{(1)}})\right) * \mathfrak{a} \\
            + \sum_{\substack{i+j=a_1+b_1 \\ q-1\mid j}}\Delta^{i,j}_{a_1,b_1} \left(y_i \left((y_{\ww{a}^{(1)}}*y_{\ww{b}^{(1)}})*\bfe(x_j) \right) \right) * \mathfrak{a} \quad \text{(by \eqref{eq-rewrite-shuffle-of-y})}
        \end{multlined} \\
        ={} &\begin{multlined}[t]
            y_{a_1}\left((y_{\ww{a}^{(1)}}*y_{\ww{b}})*\mathfrak{a}\right)
            + y_{b_1}\left((y_{\ww{a}}*y_{\ww{a}^{(1)}})*\mathfrak{a}\right)
            + y_{a_1+b_1}\left((y_{\ww{a}^{(1)}}*y_{\ww{b}^{(1)}})*\mathfrak{a}\right) \\
            + \sum_{\substack{i+j=a_1+b_1 \\ q-1\mid j}} \Delta^{i,j}_{a_1,b_1} y_i \left(\left((y_{\ww{a}^{(1)}}*y_{\ww{b}^{(1)}})*\bfe(x_j)\right)*\mathfrak{a} \right) \quad\text{(by Lemma \ref{lem-(yb)*a=y(b*a)})}.
        \end{multlined}
    \end{align*}
    For the last term, we have
    \begin{align*}
        \left(\left(y_{\ww{a}^{(1)}}*y_{\ww{b}^{(1)}}\right)*\bfe(x_j)\right) * \mathfrak{a}
        &\equiv (y_{\ww{a}^{(1)}}*y_{\ww{b}^{(1)}})*(\bfe(x_j)*\mathfrak{a}) \\
        &=(y_{\ww{a}^{(1)}}*y_{\ww{b}^{(1)}})*(\mathfrak{a}*\bfe(x_j)) \\
        &\equiv \left((y_{\ww{a}^{(1)}}*y_{\ww{b}^{(1)}})*\afk \right)*\bfe(x_j) \pmod{\mathfrak{A}}.
    \end{align*}
    Hence, by \eqref{fifth-relation},
    \begin{align*}
        &(y_{\ww{a}}x_{\ww{v}})*(y_{\ww{b}}x_{\ww{w}}) \\
        \equiv{} &\begin{multlined}[t]
            y_{a_1}\left((y_{\ww{a}^{(1)}}*y_{\ww{b}})*\mathfrak{a}\right)
            + y_{b_1}\left((y_{\ww{a}}*y_{\ww{b}^{(1)}})*\mathfrak{a}\right)
            + y_{a_1+b_1}\left((y_{\ww{a}^{(1)}}*y_{\ww{b}^{(1)}})*\mathfrak{a}\right) \\
            + \sum_{\substack{i+j=a_1+b_1 \\ q-1\mid j}} \Delta^{i,j}_{a_1,b_1} y_i \left(\left((y_{\ww{a}^{(1)}}*y_{\ww{b}^{(1)}})*\afk \right)*\bfe(x_j)\right).
        \end{multlined}\\
        \equiv{} &\begin{multlined}[t]
            y_{a_1}\left((y_{\ww{a}^{(1)}}x_{\ww{v}})*(y_{\ww{b}}x_{\ww{w}}) \right)
            + y_{b_1} \left((y_{\ww{a}}x_{\ww{v}})*(y_{\ww{b}^{(1)}}x_{\ww{w}})\right) + y_{a_1+b_1} \left((y_{\ww{a}^{(1)}}x_{\ww{v}})*(y_{\ww{b}^{(1)}}x_{\ww{w}})\right) \\
            + \sum_{\substack{i+j=a_1+b_1 \\ q-1 \mid j}} \Delta^{i,j}_{a_1,b_1} y_i \left(\left((y_{\ww{a}^{(1)}}x_{\ww{v}})*(y_{\ww{b}^{(1)}}x_{\ww{w}}) \right)* \bfe(x_j) \right)
            \pmod{\Afk}.
        \end{multlined}
    \end{align*}
\end{proof}

We are now ready to state and prove Proposition \ref{prop-ye(x)*ye(x)}.

\begin{prop}\label{prop-ye(x)*ye(x)}
    For each $a,b\in \mathbb{N}$ and indices $\ww{a},\ww{b}$, we have 
    \begin{multline*}
        \left(y_a\bfe(x_{\ww{a}})\right) * \left(y_b\bfe(x_{\ww{b}})\right)
        \equiv y_a\left(\bfe(x_{\ww{a}})*\left(y_b\bfe(x_{\ww{b}})\right) \right)
        + y_b \left( \left(y_a\bfe(x_{\ww{a}}) \right) * \bfe(x_{\ww{b}}) \right)  \\
        +y_{a+b} \left(\bfe(x_{\ww{a}})*\bfe(x_{\ww{b}}) \right)
        +\sum_{\substack{i+j=a+b \\ q-1\mid j}} \Delta^{i,j}_{a,b} y_i \left( \left(\bfe(x_{\ww{a}}) * \bfe(x_{\ww{b}}) \right) * \bfe(x_j) \right) \pmod{\mathfrak{A}}.
    \end{multline*}
\end{prop}

\begin{proof}
    We see that
    \begin{align*}
        &(y_a\bfe(x_{\ww{a}}))*(y_b\bfe(x_{\ww{b}}))  \\
        ={} &\left(\sum_{k=0}^{\dep(\ww{a})}y_ay_{\ww{a}_{(k)}}x_{\ww{a}^{(k)}}\right)*\left(\sum_{\ell=0}^{\dep(\ww{b})}y_by_{\ww{b}_{(\ell)}}x_{\ww{b}^{(\ell)}}\right) \quad\text{(by Definition \ref{df-bfe})} \\
        ={} &\sum_{k,\ell}(y_ay_{\ww{a}_{(k)}}x_{\ww{a}^{(k)}})*(y_by_{\ww{b}_{(\ell)}}x_{\ww{b}^{(\ell)}}) \\
        \equiv{}  &\begin{multlined}[t]
            \sum_{k,\ell}y_a\left((y_{\ww{a}_{(k)}}x_{\ww{a}^{(k)}})*(y_by_{\ww{b}_{(\ell)}}x_{\ww{b}^{(\ell)}})\right)
            +\sum_{k,\ell}y_b \left((y_ay_{\ww{a}_{(k)}}x_{\ww{a}^{(k)}})*(y_{\ww{b}_{(\ell)}}x_{\ww{b}^{(\ell)}}) \right) \\
            +\sum_{k,\ell}y_{a+b}\left((y_{\ww{a}_{(k)}}x_{\ww{a}^{(k)}})*(y_{\ww{b}_{(\ell)}} x_{\ww{b}^{(\ell)}})\right)
            +\sum_{k,\ell}\sum_{\substack{i+j=a+b \\ q-1\mid j}}\Delta^{i,j}_{a,b}y_i\left(\left((y_{\ww{a}_{(k)}}x_{\ww{a}^{(k)}})*(y_{\ww{b}_{(\ell)}}x_{\ww{b}^{(\ell)}})\right)*\bfe(x_j) \right) \\
            \text{(by Lemma \ref{lem-yx*yx})}
        \end{multlined} \\
        ={} &\begin{multlined}[t]
            y_a \left( \bfe(x_{\ww{a}}) * (y_b\bfe(x_{\ww{b}}) )\right) 
            + y_b \left( (y_a\bfe(x_{\ww{a}}) ) * \bfe(x_{\ww{b}}) \right)
            +y_{a+b}\left(\bfe(x_{\ww{a}})*\bfe(x_{\ww{b}})\right) \\ +\sum_{\substack{i+j=a+b \\ q-1\mid j}}\Delta^{i,j}_{a,b}y_i \left( \left( \bfe(x_{\ww{a}}) * \bfe(x_{\ww{b}}) \right) * \bfe(x_j) \right) \pmod{\mathfrak{A}} \quad\text{(by Definition \ref{df-bfe})}.
        \end{multlined}
    \end{align*}
\end{proof}

\subsection{The Final Step}

We are in the position to prove Theorem \ref{thm-main-theorem}, which is restated as follows.

\begin{thm} \label{thm-main-theorem-restate}
    For $r\ge 1$, we define $\bfE_r: \Rcal \to \Ocal(\Omega^r)$ to be the unique $\FF_p$-linear map satisfying  
    \[
    \bfE_r(1) := 1
    \quad
    \text{and}
    \quad
    \bfE_r(x_{\ww{a}}) := E_r(\ww{a};\bfz).
    \]
    Then $\bfE_r$ is an $\FF_p$-algebra homomorphism, i.e.,
    \[
    \bfE_r(x_{\ww{a}} \ast x_{\ww{b}}) = E_r(\ww{a};\bfz) E_r(\ww{b};\bfz).
    \]
\end{thm}

\begin{proof}
    We proceed by induction on the rank $r$.
    For the base case $r=1$, since $\bfE_1=\bfzeta_A$, the result follows from Theorem \ref{thm-zeta-alg-hom}.
    Let $r \geq 2$ and assume the result holds for $r-1$, i.e., $\bfE_{r-1}$ is an $\mathbb{F}_p$-algebra homomorphism.
    Since the $\Fp$-algebra $\Ocal(\Omega^r)$ is associative, the map $\bfG_r: \Ecal \to \mathcal{O}(\Omega^r)$ factors through $\mathcal{E} \to \mathcal{E}/\mathfrak{A}$, which is still denoted by $\bfG_r$.
    Then one sees by Definitions \ref{df-bfG}, \ref{df-bfe}, and \eqref{eq-restate-t-expansion} that $\bfE_r$ factors as
    \[
    \bfE_r = \bfG_r \circ \bfe: \Rcal \to \Ecal/\mathfrak{A} \to \mathcal{O}(\Omega^r).
    \]
    Hence, the result follows from Propositions \ref{prop-bfG-ring-hom} and \ref{prop-bfe-alg-hom}.
\end{proof}

\begin{rem}
    It is interesting to explore whether linear relations among multiple zeta values can be lifted to the corresponding multiple Eisenstein series.
    On the other hand, we also expect that the $q$-shuffle algebra $\Ecal$ possesses rich algebraic structures, which could provide higher insights in the study of multiple zeta values. These aspects will be investigated in future work. 
\end{rem}

\subsection*{Acknowledgments}

The authors thank Chieh-Yu Chang for suggesting this project, as well as for his careful reading of the manuscript and insightful remarks, which have greatly contributed to the development of this work.
They further express their sincere gratitude to \textit{The 17th MSJ-SI Developments of Multiple Zeta Values Conference 2025} for providing an excellent environment that fostered fruitful discussions.
The authors extend special thanks to the National Science and Technology Council for its financial support over the past few years.

\printbibliography

\end{document}